\documentclass[english]{amsart}

\usepackage{amsmath,amssymb,enumerate}

\usepackage{babel}
\usepackage{amstext}
\usepackage{amsmath}
\usepackage{amsfonts}
\usepackage{latexsym}
\usepackage{ifthen}
\usepackage{xypic}

\usepackage{amssymb}
\usepackage{amscd}
\usepackage{exscale}
\usepackage{units}

\xyoption{all}
\newcommand{\id}{{\rm id}}

\newcommand{\lra}{\longrightarrow}

\newcommand{\im}{{\rm im}}

\newcommand{\sheafhom}{\mathcal{H}\kern -.5pt om}

\newcommand\sE{{\mathcal E}}

\newcommand\sF{{\mathcal F}}

\newcommand\sI{{\mathcal I}}

\newcommand\sL{{\mathcal L}}
\newcommand\sO{{\mathcal O}}
\newcommand\sS{{\mathcal S}}

\newcommand\sT{{\mathcal T}}

\newcommand\sHom{{\mathcal Hom}}
\newcommand\sExt{{\mathcal Ext}}

\newcommand\sEnd{{\mathcal End}}

\newcounter{lemma}

\newtheorem{lemma1}[lemma]{\setcounter{equation}{0}}

\newenvironment{lemma}{\begin{lemma1}{\bf Lemma.}}{\end{lemma1}}

\newenvironment{theorem}{\begin{lemma1}{\bf Theorem.}}{\end{lemma1}}
\newenvironment{question}{\begin{lemma1}{\bf Question.}}{\end{lemma1}}

\newenvironment{proposition}{\begin{lemma1}{\bf Proposition.}}{\end{lemma1}}

\newenvironment{corollary}{\begin{lemma1}{\bf Corollary.}}{\end{lemma1}}

\newenvironment{remark}{\begin{lemma1}{\bf Remark.}\rm}{\end{lemma1}}
\newenvironment{definition}{\begin{lemma1}{\bf Definition.}}{\end{lemma1}}

\newenvironment{notation}{\begin{lemma1}{\bf Notation.}}{\end{lemma1}}

\newenvironment{Induction Step}{\begin{lemma1}{\bf Induction Step.}}{\end{lemma1}}
\newenvironment{Proof of Theorem 1.2}{\begin{lemma1}{\bf Proof of Theorem 1.2.}}{\end{lemma1}}

\newenvironment{ps-example}{\begin{lemma1}{\bf Pseudo-Example.}\rm}{\end{lemma1}}

\newenvironment{problem}{\begin{lemma1}{\bf Problem.}\rm}{\end{lemma1}}

\begin{document}
\title{Deformations of the tangent bundle of projective manifolds}


\author{Thomas Peternell} 

\address{Thomas Peternell, Mathematisches Institut, Universit\"at Bayreuth, 95440 Bayreuth, 
Germany}
\email{thomas.peternell@uni-bayreuth.de}
%\date{today} 

\maketitle

\begin{abstract} We investigate when the tangent bundle of a projective manifold has a non-trivial first order
(or positive-dimensional) deformation. This leads to a new conjectural characterization of the complex projective
 space.

\end{abstract} 
\tableofcontents

%\large

\section{Introduction} 

In the early 1990's, physicists asked to compute the dimension of the space $H^1(X,T_X \otimes \Omega^1_X)$ for a Calabi-Yau threefold 
$X$; see  \cite{Dis89}, \cite{EH90}.  They computed the dimension for some special complete intersections. A few years later the problem was taken up by D.Huybrechts \cite{Huy95};
he proved in particular that $H^1(X,T_X \otimes \Omega^1_X)  \ne 0$ for three-dimensional Calabi-Yau complete intersections in projective spaces. 
In terms of deformation theory, this space parametrizes first order deformations of the holomorphic tangent bundle $T_X$.

In this paper we propose to systematically study deformations of the tangent bundle $T_X$ of any compact complex manifold $X$, up to deformations of the
form $T_X \otimes \sL$, where $\sL$ is a deformation of the trivial line bundle $\sO_X$. To be precise, we introduce the following notation.

\begin{notation} {\rm  Let $X$ be a compact complex manifold. 
 We say that $T_X$ has a {\it genuine first order deformation} if $T_X$ has a deformation $\sE$ over the double point $D$ which
is not induced by a deformation of $\sO_X$, i.e., $\sE \not \simeq p^*T_X \otimes \sL$ with $\sL$ a deformation of $\sO_X$ over $D$; here $p: X \times D \to X$ denotes the projection. }
\end{notation} 

To have a genuine first oder deformation is equivalent to saying that 
\begin{equation} \label{eq1-1}  h^1(X,T_X \otimes \Omega^1_X)  > q(X) := h^1(X,\sO_X).\end{equation} 
Note that $T_X$ has a non-trivial first order deformation if and only if the morphism $\pi: \mathbb P(T_X) \to X$ has a first order deformation with fixed target $X$
which is not trivial, i.e., not constant. 

In the same way we define genuine deformations of $T_X$ over a positive-dimen\-sional parameter space. The obstructions to lifting a first order deformation in the sense of (1) 
ly in the space 
\begin{equation} \label{eq1a}  H^2(X,T_X \otimes \Omega^1_X)/H^2(X,\sO_X) \simeq H^2(\mathbb P(T_X), T_{\mathbb P(T_X)/X}). \end{equation} 

Thus $T_X$ has a non-obstructed genuine deformation provided
\begin{equation} \label{eq1b} 
h^1(X,T_X\otimes \Omega^1_X) - h^2(X,T_X \otimes \Omega^1_X) - q(X) + h^2(X,\sO_X) > 0. \end{equation}

\begin{remark}  {\rm   Let ${\sEnd}_0(T_X)$ denote the sheaf of traceless endomorphisms of $T_X$, hence
$$ T_X \otimes \Omega^1_X \simeq {\sEnd}_0(T_X) \oplus \sO_X$$ and a fortiori
$$ H^q(X,T_X \otimes \Omega^1_X) = H^q(X,{\sEnd}(T_X)) =  H^q(X,{\sEnd}_0(T_X))  \oplus H^q(X, \sO_X))$$
for all $q$.
Then Equation (\ref{eq1-1}) is equivalent to 
$$ h^1(X,{\sEnd}_0(T_X)) \ne 0.$$}
whereas Equation (\ref{eq1b}) is equivalent to 
$$ h^1(X,{\sEnd}_0(T_X)) - h^2(X, {\sEnd}_0(T_X)) > 0.$$ 
\end{remark}

The theme of this paper is now the following 

\begin{problem} Let $X$ be a projective (compact K\"ahler, or simply compact) manifold.  Give necessary and sufficient conditions such that  $T_X$  has a genuine first-order deformation
resp. a genuine non-obstructed deformation. 
\end{problem}

In dimension $2$ we show 

\begin{theorem} Let $X$ be a compact K\"ahler surface. Then $T_X$ does not have a genuine first order deformation if 
and only if either $X \simeq \mathbb P_2$ or if $X$ is a ball quotient (so $c_1^2(X) = 3c_2(X)$) with $H^0(X,S^2\Omega^1_X) = 0$. 
\end{theorem} 

The case that $X$ is of general type has already been shown in \cite{CWY03}. 
It would certainly be interesting to classify the two-dimensional ball quotients $X$ with  $H^0(X,S^2\Omega^1_X) = 0$. 

In dimension larger than one however, things get much more complicated, even in dimension $3$. Of course, the tangent bundle of the complex projective space is rigid in any dimension. 
In dimension $3$ we prove - among other things -

\begin{theorem} Let $X$ be a three-dimensional compact complex manifold. 
Then $T_X$ has a genuine first order deformation provided one of the following conditions is satisfied. 
\begin{enumerate} 
\item $X$ is a Fano threefold, different from $\mathbb P_3$.
\item Suppose $X$ is rationally connected and that $c_1(X)^3 \leq 63$.
\item $X$ is the blow-up of the smooth threefold $Y$ in a point or a smooth curve $C$. Suppose that  $h^0(Y,T_Y  \otimes \Omega^1_Y) = 1$, e.g., $T_Y$ is stable for some polarization, and in case of
the blow-up of $C$, additionally that $-K_Y \cdot C \geq 2$.
\item $X$ is a $\mathbb P_2$-bundle over a smooth curve. 
\item $X = \mathbb P(\sF) $ with a semi-stable rank two bundle over a smooth K\"ahler surface $S$ with $H^1(S,\sO_S) = 0$.
\item $X$ is a Calabi-Yau threefold and $\varphi: X \to Y$ the contraction of an irreducible divisor $E$ to a point or a curve with $Y$ projective
or $\varphi: X \to Y$ a small contraction, with a few exceptions, given in Theorem \ref{cont2} and Theorem \ref{thm:ruled}. 
\end{enumerate}
\end{theorem} 

It is also easy to see that the tangent bundles of smooth hypersurfaces $X \subset \mathbb P_{n+1}$ of degree $d \geq 2$ and dimension $n \geq 2$ 
have genuine first order deformations. If $X$ is a ball quotient of dimension at least three, then Siu \cite{Siu91} has shown that $T_X$ has no genuine deformation.
Further, if $X $ is a product $X = X_1 \times X_2$ with $X_j$ ball quotients with $q(X_j) = 0$,
then $T_X$ again has no genuine first order deformation. This procedure can also be
iterated.

All these considerations lead to the following 

\begin{question} \label{Q} Suppose $X$ is a compact K\"ahler manifold. Then $T_X$ has no genuine first order deformation if and only if $X$ is one of the following. 
\begin{itemize}
\item $X \simeq \mathbb P_n$ 
\item $X$ is a ball quotient of dimension at least three
\item $X$ is a twodimensional ball quotient such that $H^1(X,\sEnd_0(T_X))  = 0 $
\item $X$ is a product (with at least two factors) 
$$ X = \Pi_j X_j$$
with $X_j$ 
ball quotients such that $q(X_j) = 0$ and $H^1(X_j, \sEnd_0(T_{X_j})) = 0 $ for all $j$. 

\end{itemize} 
\end{question}

{\bf Acknowledgements} Many thanks go to Ulrike Peternell for carefully reading the first four parts of the paper and for suggesting many improvements.
Also I would like to thank Alan Huckleberry and Stefan Nemirovski for very helpful discussions and Jun-Muk Hwang for pointing out the references \cite{CWY03} and \cite{Siu91}. 
Finally, I would like to thank the referee for pointing out an error in section 5 and many other very valuable suggestions.

\section{Some general results}
\setcounter{lemma}{0}

In this section we prove some general results with a focus on blow-ups.

\begin{proposition} \label{prop:bir0}
Let $X$ be a compact complex manifold. Let $E \subset X$ be a smooth irreducible divisor. Assume that 
\begin{enumerate}
\item $H^0(E, T_X \vert E) \ne 0.$
\item $H^0(X,\Omega^1_X \otimes T_X) = H^0(X, \Omega^1_X(\log  E) \otimes T_X), $
\end{enumerate} 
Then $H^1(X,\Omega^1_X \otimes T_X) \ne 0. $ 
\end{proposition}

\begin{proof} The residue sequence tensorized with $T_X$ reads 
$$ 0  \to \Omega^1_X \otimes T_X \to \Omega^1_X(\log E) \otimes T_X \to T_X \vert E \to 0. $$
Taking cohomology and using our assumption (2) we obtain an injection 
$$ H^0(E,T_X \vert E) \to H^1(X,T_X \otimes \Omega^1_X) $$
Hence assumption (1) gives the claim.
\end{proof} 

In the same way, we have 

\begin{proposition} \label{prop:bir0A}
Let $X$ be a compact complex manifold. Let $E \subset X$ be a smooth irreducible divisor. Assume that 
\begin{enumerate}
\item $H^0(E, T_X \vert E) \ne 0.$
\item $h^0(X,\sEnd_0(T_X)) = h^0(X, (\Omega^1_X(\log  E) \otimes T_X) /\sO_X)$.
\end{enumerate} 
Then $H^1(X, {\sEnd}_0(T_X) \ne 0$.
\end{proposition}

\begin{proof} 
Dividing by the trivial summand of $\Omega^1_X \otimes T_X$, the residue sequence induces an exact sequence
$$ 0 \to  {\sEnd}_0(T_X) \to \mathcal R := (\Omega^1_X(\log E) \otimes T_X) / \sO_X \to 
T_X \vert E \to 0.$$
Then we conclude as before. 
\end{proof}

Since $ 0 \ne H^0(E,T_E) \subset H^0(E,T_X \vert E)$, we obtain

\begin{corollary} 
 \label{prop:bir1}
Let $X$ be a compact complex manifold. Let $E \subset X$ be a smooth irreducible divisor. Assume that 
\begin{enumerate}
\item $H^0(E,T_E) \ne 0, $
\item $H^0(X,\Omega^1_X \otimes T_X) = H^0(X, \Omega^1_X(\log  E) \otimes T_X)$.

\end{enumerate} 
Then $H^1(X,T_X \otimes \Omega^1_X) \ne 0. $ 
\end{corollary}

\begin{corollary} 
 \label{prop:bir1A}
Let $X$ be a compact complex manifold. Let $E \subset X$ be a smooth irreducible divisor. Assume that 
\begin{enumerate}
\item $H^0(E,T_E) \ne 0, $
\item $h^0(X,\sEnd_0(T_X)) = h^0(X, \Omega^1_X(\log E)  \otimes T_X/ \mathcal O_X )$.

\end{enumerate} 
Then $H^1(X,{\sEnd}_0(T_X)) \ne 0. $ 
\end{corollary}

 \begin{remark} The condition
 $$H^0(X,\Omega^1_X \otimes T_X) = H^0(X, \Omega^1_X(\log  E) \otimes T_X),$$
 certainly holds provided
$$ \dim H^0(X, \Omega^1_X(\log E) \otimes T_X) = 1.$$
\end{remark}

\begin{proposition}  \label{prop:blow-up} Let $X$ be a compact complex $n$-dimensional manifold, $x_0 \in X$ and $\pi: \hat X \to X$ the blow up of $X$  at $x_0$. 
Then 
\begin{enumerate}
\item There is an exact sequence 
\begin{equation} \label{eqA} 0 \to \pi_*(T_{\hat X}\otimes \Omega^1_{\hat X}) \to T_X \otimes \Omega^1_X \to Q \to 0 \end{equation}
with a sheaf $Q$ supported on $x_0$ of length $n^2-1$. 
\item
$h^1(\hat X, T_{\hat X} \otimes \Omega^1_{\hat X}) \geq h^1(X,T_X \otimes \Omega^1_X)$
with strict inequality if $h^0(X,T_X \otimes \Omega^1_X) < n^2$;
\item
$ h^2(\hat X, T_{\hat X} \otimes \Omega^1_{\hat X}) \leq h^2(X,T_X \otimes \Omega^1_X)$;
\item $h^1(\hat X, {\sEnd}_0(T_{\hat X}))  \geq h^1(X, {\sEnd}_0(T_X)) $
with strict inequality if $h^0(X, T_X \otimes \Omega^1_X) < n^2$;
\item $ h^2(\hat X, {\sEnd}_0(T_{\hat X})) \leq h^2(X, {\sEnd}_0(T_X))$.
\end{enumerate} 
\end{proposition} 

\begin{proof} 
Set $E = \pi^{-1}(x_0)$. We clearly have an exact sequence
$$ 0 \to \pi_*(T_{\hat X} \otimes \Omega^1_{\hat X}) \to T_X \otimes \Omega^1_X \to Q \to 0  $$
with a sheaf
sheaf $Q$ which is zero outside $x_0$. 
Tenorize the exact sequence 
$$ 0 \to \pi^*(\Omega^1_X) \to \Omega^1_{\hat X} \to \Omega^1_{\hat X/X} \to 0 $$
by $T_{\hat X}$, take $\pi_*$ and use $R^1\pi_*(T_{\hat X}) = 0 $ to obtain the exact sequence 
$$ 0 \to \pi_*(T_{\hat X}) \otimes \Omega^1_X \to \pi_*(T_{\hat X} \otimes \Omega^1_{\hat X}) \to \pi_*(T_{\hat X} \vert E \otimes \Omega^1_{E}) \to 0.$$
Further, we have an exact sequence
$$ 0 \to \pi_*(T_{\hat X}) \otimes \Omega^1_X  \to T_X \otimes \Omega^1_X \to \mathbb C^n_{x_0} \otimes \Omega^1_X \to 0.$$
Then we obtain the following commutative diagram 

\begin{equation} \label{BigCommDiagramm}
\xymatrix{
& & 0 \ar[d] & 0 \ar[d] & & & \\ 
& 0 \ar[r] & \pi_*(T_{\hat X}) \otimes \Omega^1_X  \ar[d] \ar[r]  &  \pi_*(T_{\hat X} \otimes \Omega^1_{\hat X}) \ar[d] \ar[r] & \pi_*(T_{\hat X} \otimes \Omega^1_E)  \ar[r] & 0 \\
&  & T_X \otimes \Omega^1_X \ar[r]^{=} \ar[d]   & T_X \otimes \Omega^1_X \ar[d]  & &   \\
0  \ar[r] & \ker \lambda \ar[r] & \mathbb C_{x_0}^n \otimes \Omega^1_X \ar[r]^{\lambda} \ar[d] & Q \ar[d] \ar[r] &0  &   \\
& & 0 & 0 &  & & 
}
\end{equation}

A diagram chase shows that 
$$ \ker \lambda \simeq   \pi_*(T_{\hat X} \vert E \otimes \Omega^1_{E}) ,$$ 
yielding an exact sequence
$$ 0 \to \pi_*(T_{\hat X} \vert E \otimes \Omega^1_{E}) \to  \mathbb C^n_{x_0} \otimes \Omega^1_X \to Q \to 0.$$
Observe finally 
\begin{equation} \label{coh}   \pi_*(T_{\hat X} \vert E \otimes  \Omega^1_{E}) \simeq \mathbb C_{x_0}. \end{equation} 
In fact, use the exact sequence 
$$ 0 \to T_E \otimes \Omega^1_E \to T_{\hat X} \vert E \otimes \Omega_E^1 \to N_E \otimes \Omega^1_E \to 0 $$
to see that 
$$  h^0(E, T_{\hat X} \vert E \otimes \Omega^1_E) = 1 $$
and
$$ H^0(E, N_E^{*\otimes \mu} \otimes T_{\hat X} \vert E \otimes \Omega^1_E) = 0$$
for all $\mu \geq 1$, which yields (\ref{coh}). 
Then Assertion (1) follows by (\ref{coh}). \\
Assertions (2) and (3) then  follow immediately by taking cohomology and using the Leray spectral 
sequence. Finally, (4) and (5) are immediate from (2) and (3), since $H^q(\hat X,\sO_{\hat X}) = H^q(X,\sO_X)$ for all $q$. 

\end{proof} 

\begin{corollary} \label{cor:blow-up}  Let $X$ be a compact complex $n$-dimensional manifold, $x_0 \in X$ and $\pi: \hat X \to X$ the blow up of $X$  at $x_0$. 
Suppose that $T_X$ has a genuine first order (resp. genuine non-obstructed) deformation. Suppose further that $h^0(X,T_X \otimes \Omega^1_X) < n^2$. 
Then  $T_{\hat X}$ has a genuine first order (genuine non-obstructed) deformation.
\end{corollary}

If  $T_X$ is simple, we can say more:

\begin{theorem} \label{thm:gendiv} 
Let $X$ be a  compact complex manifold
such that $$H^0(X,\Omega^1_X \otimes T_X ) \simeq \mathbb C;$$
e.g., $T_X$ is stable with respect to a Gauduchon metric. 
Let either $\pi: \hat X \to X$ be the blow-up of a point in $X$ or $\dim X \geq 3$ and  $\pi: \hat X \to X$ be the blow-up
of a smooth rational curve $C \subset X$. Then 
$$ H^1(\hat X, {\sEnd}_0(T_{\hat X})) \ne 0. $$
\end{theorem} 

\begin{proof} We restrict ourselves to the case of the blow-up of a curve; the point blow-up being completely analogous. 
We verify the conditions of  Corollary \ref{prop:bir1A}, applied to the exceptional divisor $E = \pi^{-1}(C)$
and use
again the notation 
$$ \mathcal R = (\Omega^1_{\hat X}(\log E) \otimes T_{\hat X}) / \sO_{\hat X}.$$ 
Since $E = \mathbb P(N^*_{C/X}) $, the relative Euler sequence and the splitting of $N_{C/X}$ yields
$$ h^0(E,T_{E/C}) = h^0(C,N^*_{C/X} \otimes N_{C/X}) - 1 \geq 1,$$
thus $H^0(E,T_E) \ne 0$. 
As to the second condition, we first observe the following chain of inclusions and equations
$$ H^0(\hat X \setminus E,\mathcal R) = H^0(\hat X \setminus E, {\sEnd}_0(T_{\hat X)}) = H^0(X \setminus C, {\sEnd}_0(T_X)) = $$
$$ = H^0(X,{\sEnd}_0(T_X)),$$ 
the last equation coming from Riemann's extension theorem.\\
By our assumption,  $H^0(X,{\sEnd}_0(T_X)) = 0$, hence $H^0(\hat X \setminus E, \mathcal R) = 0$. 
Now $\mathcal R$ is torsion free; in fact, otherwise $\sO_X $ would not be saturated in $\Omega^1_X(\log E) \otimes T_X$, thus $\sO_X \to \Omega^1_X(\log E) \otimes T_X$
would vanish along $E$, which is clearly not the case. 
Consequently, 
$H^0(X,\mathcal R) \to H^0(X \setminus E,\mathcal R) $ is injective, hence we conclude.

%$$ H^0(\hat X, \Omega^1_{\hat X} \otimes T_{\hat X}) \subset  H^0(\hat X, \Omega^1_{\hat X}(\log E) \otimes T_{\hat X}) \subset H^0(\hat X \setminus E, \Omega^1_{\hat X} \otimes T_{\hat X}) = $$
%$$ = H^0(X \setminus C, \Omega^1_X \otimes T_X) = H^0(X, \Omega^1_X \otimes T_X),$$
%the last equation coming from Riemann's extension theorem.
%On the other hand, 
%$$  H^0(\hat X, \Omega^1_{\hat X} \otimes T_{\hat X})  = H^0(X, \Omega^1_X \otimes T_X) \simeq \mathbb C.$$
%In fact, $$H^0(\hat X, \Omega^1_{\hat X} \otimes T_{\hat X}) = {\rm Hom}(T_X,T_X) \ne 0. $$ 
%Hence we have equality everywhere, in particular, the first inclusion is an equality. Using again the notation 
%$$ Q = \Omega^1_{\hat X}(\log E) \otimes T_{\hat X} / \sO_{\hat X}),$$ 
%we have an exact sequence 
%$$ 0 \to \sO_X \to \pi_*(\Omega_{\hat X}(\log E) \otimes T_{\hat X}) \to \pi_*(Q) \to 0 ,$$
%hence
%$$ h^0(\hat X, \Omega_{\hat X}(\log E) \otimes T_{\hat X}) = h^0(\hat X...)$$
\end{proof} 

\begin{corollary}  \label{corelli} Theorem \ref{thm:gendiv} remains true for curves $C$ of genus $g \geq 1,$ provided $H^0(T_E) \ne 0.$ 
If $g \geq 2$, this is equivalent to $$h^0(N_C \otimes N_C^*) \geq 2,$$
i.e., $N_C$ is not simple. 
In case $g = 1$, we might also have $h^0(N_C \otimes N_C^*) = 1 $
and the vector field on $C$ lifts to $E.$ 
\end{corollary} 

\begin{remark} Instead of assuming $T_X$ to be simple in Theorem \ref{thm:gendiv}, it suffices to assume that 
$$ H^0(\hat X, T_{\hat X} \otimes \Omega^1_{\hat X}) = H^0(X,T_X \otimes \Omega^1_X). $$

\end{remark}

\section{Surfaces} 
\setcounter{lemma}{0}

We start by some general calculations. 

\begin{proposition} \label{prop:RR2} 
Let $X$ be a smooth compact complex surface.  Then 
\begin{enumerate}
\item $\chi(X,T_X \otimes \Omega^1_X) = \frac{1}{3} \bigl (4c_1^2(X) - 11c_2(X) \bigr)$
\item $h^1(X,T_X \otimes \Omega^1_X) = $ \\
$ \frac{1}{3} \bigl(-4c_1^2(X) +11c_2(X) \bigr) + h^0(X,T_X \otimes \Omega^1_X)  + 
 h^0(X,(\Omega^1_X)^{\otimes 2})$ 
\item $h^1(X,T_X \otimes \Omega^1_X) = $\\
$  \frac{5}{4} \bigl(3 c_2(X)  - c_1^2(X)) + q(X) + h^0(X,{\sEnd}_0(T_X)) + h^0(X,S^2\Omega^1_X).$
\item $h^1(X,{\sEnd}_0(T_X)) - h^2(X,{\sEnd}_0(T_X)) = $ \\
$  \frac{5}{4} (3 c_2(X)  - c_1^2(X) ) + h^0(X,T_X \otimes \Omega^1_X) - 1.$
\item If $c_1^2(X) < 3c_2(X)$, then $h^1(X,{\sEnd}_0(T_X)) - h^2(X,{\sEnd}_0(T_X)) > 0$.
\end{enumerate}

\end{proposition} 

\begin{proof} (1) follows from Riemann-Roch, since $c_1(T_X \otimes \Omega^1_X) = 0$ and $c_2(T_X \otimes \Omega^1_X) = 4c_2(X) - c_1^2(X).$ \\
(2)  is a consequence of (1), using 
$$ H^2(X,T_X \otimes \Omega^1_X) \simeq H^0(X,T_X \otimes \Omega^1_X \otimes K_X) = H^0(X,(\Omega^1_X)^{\otimes 2}).$$
For (3), we apply (2), observe that $$(\Omega^1_X)^{\otimes 2} \simeq S^2\Omega^1_X \oplus K_X$$
and use $$ h^0(X,K_X) = h^2(X,\sO_X) = \chi(X,\sO_X)  - 1 + q(X) = \frac{1}{12} \bigl (c_1^2(X) + c_2(X)\bigr) - 1 + q(X). $$
As to (4), we have, using (1),  
$$h^1(X,{\it End}_0(T_X)) - h^2(X,{\it End}_0(T_X)) = h^1(X,T_X \otimes \Omega^1_X) - h^2(X,T_X \otimes \Omega^1_X) - q(X) + $$
$$ +  h^2(X,\sO_X) =  -  \chi(X,T_X \otimes \Omega^1_X) + h^0(X,T_X \otimes \Omega^1_X) + \chi(X,\sO_X) - 1 = $$
$$ =  \frac{45}{12} c_2(X)  - \frac{15}{12}c_1^2(X) + h^0(X,T_X \otimes \Omega^1_X) - 1.$$
This yields claim (4), and (5) follows from (4). 

\end{proof}

\begin{corollary} \label{cor1} Let $X$ be a compact complex surface. If $c_1^2(X) < 3 c_2(X) $, then $T_X$ has a genuine non-obstructed deformation. 
\end{corollary} 

\begin{theorem} \label{thm1} Let $X$ be a compact complex K\"ahler surface. Then the following assertions are equivalent.
\begin{enumerate}
\item $T_X$ has a genuine first order deformation.
\item  $X \not \simeq \mathbb P_2$ and
$X$ is not a ball quotient with $H^0(X,S^2\Omega^1_X) = 0$. 

\end{enumerate}

\end{theorem}

As already mentioned, surfaces of general type have been treated in \cite{CWY03}.

\begin{proof}  First that if $X = \mathbb P_2$ or if $X$ is a ball quotient with $H^0(X,S^2\Omega^1_X) = 0$, then by Proposition \ref{prop:RR2}, $T_X$ has no genuine first order deformations. Hence Assertion (1) implies
Assertion (2). \\
To prove the converse, we note first that by Proposition \ref{prop:blow-up}, we may assume $X$ to be minimal and by Corollary \ref{cor1} that 
$c_1^2(X) \geq 3c_2(X) $. The Miyaoka-Yau inequality and surface classification gives $c_1^2(X) = 3c_2(X)$, unless $X$ is a ruled surface over a curve of genus at least two. More specifically,
again by classification, 
$X$ is one of the following 
\begin{enumerate}
\item $X = \mathbb P_2$;
\item $X$ is a ball quotient;
\item $\kappa (X) = 1 $ and $c_2(X) = 0$;
\item $X$ is a torus or hyperelliptic; 
\item $X$ is a ruled surface over a curve $B$ of genus $g = g(B) \geq 1$. 

\end{enumerate} 

Case (4) is immediately settled by Proposition \ref{prop:RR2}(3). \\
(3) Assume  that $\kappa (X) = 1$ and $c_2(X) = 0$. Let $f: X \to B$ be the Iitaka fibration. 
Suppose $T_X$ does not have a genuine first order deformation, then we have $ h^1(X,T_X \otimes \Omega^1_X) = q(X)$. Hence $H^0(X,S^2\Omega^1_X) = 0$ by Proposition \ref{prop:RR2}(3). 
Hence $q(X) = 0$, so that $\chi(X,\sO_X) > 0$, contradicting $c_1^2(X) = c_2(X) = 0$. \\
(5) Let $\pi: X \to B$ denote a ruling over the curve $B$. 
 Since $\chi(X,T_X \otimes \Omega^1_X) = 4(g-1)$, we may assume $g \geq 2$. 
By Proposition \ref{prop:RR2}(2), 
$$ h^1(X,T_X \otimes \Omega^1_X ) = 4(1-g) + h^0(X,T_X \otimes \Omega^1_X) + h^0(X, (\Omega^1_X)^{\otimes 2}) \geq $$
$$ \geq 4(1-g) + h^0(X,T_X \otimes \Omega^1_X) + h^0(B,2K_B) = 1-g + h^0(X,T_X \otimes \Omega^1_X).$$ 
We will now prove that 
\begin{equation} \label{eq:ruled} h^0(X,T_X \otimes \Omega^1_X) >  3(g-1). \end{equation}
This yields $$ h^1(X,T_X \otimes \Omega^1_X) >  q(X) = g,$$
which was to be proved. 
In order to show (\ref{eq:ruled}), we consider the subbundles
$$ T_X \otimes \pi^*(\Omega^1_B) \subset T_X \otimes \Omega^1_X$$ 
and $$ T_{X/B} \otimes \pi^*(\Omega^1_B) \subset T_X \otimes \pi^*(\Omega^1_B).$$
Write $X = \mathbb P(\sE)$ with a rank two bundle $\sE$ on $B$. Then 
$$ T_{X/B} \simeq \sO_{\mathbb P(\sE)}(2) \otimes \pi^*(\det \sE^*).$$
Hence 
$$ h^0( T_{X/B} \otimes \pi^*(\Omega^1_B)) = h^0(B,S^2\sE \otimes \det \sE^* \otimes \Omega^1_B) \geq $$
$$ \chi(S^2\sE \otimes \det \sE^* \otimes \Omega^1_B) = 3(g-1).$$
The last equation is Riemann-Roch for the rank three bundle $S^2(\sE) \otimes \det \sE^*$, recalling that $c_1(S^2(\sE) \otimes \det \sE^*) = 0$. 
Thus $$ h^0(X,T_X \otimes \Omega^1_X) > h^0(T_X \otimes \pi^*(\Omega^1_B)) \geq h^0( T_{X/B} \otimes \pi^*(\Omega^1_B)) \geq 3(g-1),$$
proving (\ref{eq:ruled}). The strictness of the first equality comes from the fact that the identity map in $H^0(X,T_X \otimes \Omega^1_X)$ is not induced by an element
of $H^0(X,T_{X/B} \otimes \pi^*(\Omega^1_B)$.  \\
In summary, if $T_X$ does not have a genuine deformation, then  $X$ is either $\mathbb P_2$ or a ball quotient, and in the latter case, necessarily $H^0(X,S^2\Omega^1_X) = 0 $ by Proposition \ref{prop:RR2}(3).  \\
\end{proof} 

\begin{remark} {\rm There are non-K\"ahler surfaces whose  tangent bundles have no genuine first order deformations. 
For example, let $X $ be an Inoue surface of type $S_N^+$, \cite{In74}. These are exactly the surfaces with $\kappa (X) = - \infty$, 
having no curves and $c_1^2(X) = c_2(X) = 0$. Moreover 
$$ h^0(X,T_X) = h^1(X,T_X) = 1,$$
\cite[Prop.3]{In74}. 
Let $v \in H^0(X,T_X)$ be a non-zero vector field. Then $v$ has no zeroes and induces an exact sequence
$$ 0 \to \sO_X \to T_X \to \sO_X(-K_X) \to 0.$$
From Riemann-Roch, we have $\chi(\Omega^1_X) = 0$, hence $H^1(X,\Omega^1_X)$ (since $ H^0(X,\Omega^1_X ) = H^2(X,\Omega_X^1) = 0$). 
Thus taking cohomology of the preceeding exact sequence, tensorized by $\Omega^1_X$, we obtain
$$ 0 = H^1(X,\Omega^1_X) \to H^1(X,T_X \otimes \Omega^1_X) \to H^1(X,T_X) \simeq \mathbb C, $$
and therefore
$$ h^1(X,T_X \otimes \Omega^1_X ) = 1 = h^1(X,\sO_X).$$ 
}
\end{remark} 

Theorem \ref{thm1} has the following partial strengthening

\begin{theorem} \label{thm2}
Let $X$ be a compact K\"ahler surface. Assume that $X$ is neither $\mathbb P_2$, nor a ball quotient nor of the form $\mathbb P(\sE)$ with $\sE$ 
a stable locally free sheaf of rank two over an elliptic curve. 
Then $T_X$ has a genuine non-obstructed deformation.
\end{theorem}

\begin{proof} By Proposition \ref{prop:blow-up}, we may assume $X$ minimal. 
What remains to be proved is the following. Suppose $X$ is one of the following.
\begin{enumerate} 
\item $\kappa (X) = 1$ and $c_2(X) = 0$
\item $X$ is a torus or hyperelliptic
\item $X$ is a ruled surface over a curve $B$ of genus $g = g(B) \geq 1$, but not of the form $\mathbb P(\sE)$ with $\sE$ a stable locally free sheaf of rank two
over an elliptic curve. 
\end{enumerate}
Then $T_X$  has a non-trivial non-obstructed deformation. \\
(1) Assume first that $\kappa (X) = 1$. By Proposition \ref{prop:RR2}(4), it suffices to show that 
$$ h^0(X,T_X \otimes \Omega^1_X) \geq 2.$$ 
To do this, consider the Iitaka fibration $f: X \to B$.  Since $c_2(X) = 0$, the only singular fiber of $f$ are multiples
$m_i F_i$ of elliptic curves $F_i$; write $D = \sum (m_i-1)F_i$.
The elliptic bundle formula now reads
$$ K_X = f^*(K_B \otimes \sL) \otimes \sO_X(D) $$
with a torsion line bundle $\sL$. 
Further,
there is an exact sequence
$$ 0 \to f^*(K_B) \otimes \sO_X(D) \to \Omega^1_X \to K_{X/B} \otimes \sO_X(-D) \to 0,$$ 
and therefore an inclusion
$$ T_X \otimes f^*K_B \otimes \sO_X(D) \subset T_X \otimes \Omega^1_X.$$
Thus it suffices to show 
$$ H^0(X,T_X \otimes f^*K_B \otimes \sO_X(D)) \ne 0.$$
Indeed, a non-zero element in the space is a morphism $\Omega^1_X \to f^*K_B \otimes \sO_X(D)$.
Composing with the inclusion $f^*K_B \otimes \sO_X(D) \to \Omega^1_X$ yields a morphism $\Omega^1_X \to \Omega^1_X$ which is not
a multiple of ${\rm id}$. 
Dualizing the last exact sequence yields an inclusion 
$$ f^*(K_B \otimes \sL^*) \otimes \sO_X(D) \simeq  K_{X/B}^* \otimes \sO_X(2D) \otimes f^*(K_B) \to T_X \otimes f^*K_B \otimes \sO_X(D).$$
Now $$ H^0(X,f^*(K_B \otimes \sL^*) \otimes \sO_X(D)) = H^0(B, K_B \otimes \sL^*) \ne 0,$$
unless $g = 1$ and $\sL$ not trivial, we are done except for this special case. Here we perform a finite \'etale base change $\tilde B \to B$  to trivialize $\sL$  and set $\tilde X = 
X \times_B \tilde B$ with projection $\tilde f: \tilde X \to \tilde B$. Then the associated line bundle $\tilde \sL$ and therefore 
$$ h^0(\tilde X, T_{\tilde X} \otimes \Omega^1_{\tilde X}) \geq 2.$$
Thus there exists a morphism 
$$ \lambda: \Omega^1_{\tilde X} \to \Omega^1_{\tilde X},$$
which is not a multiple of the identity. Let $\mu: \tilde X \to X$ be the projection and consider
$$ \mu_*(\lambda): \mu_*(\Omega^1_{\tilde X} ) \to \mu_*(\Omega^1_{\tilde X}).$$ Via  the decomposition
$$ \mu_*(\Omega^1_{\tilde X}) = \mu_*\mu^*(\Omega^1_X) = \Omega^1_X \otimes \mu_*(\sO_{\tilde X}) = \Omega^1_X \otimes \bigoplus_{j=0}^{-s} \sL^j,$$ 
there exists a number $k$ and a non-zero morphism
$$ \psi:  \Omega^1_X  \to \Omega^1_X \otimes \sL^k.$$ 
We aim to prove that $k = 0$; hence we obtain an endomorphism of $\Omega^1_X$ which is not a multiple of the identity and conclude. 
Using the cotangent sequence, which now reads
$$ 0 \to \sO_X(D) \otimes \Omega^1_X \to K_X \otimes \sO_X(-D) \to 0,$$ 
the morphism $\psi$ induces by composition a morphism 
$$ \psi_1:  \sO_X(D) \to K_X \otimes \sO_X(-D) \otimes f^*(\sL^k) = f^*(\sL \otimes \sL^k). $$
If $\psi_1 \ne 0$, then $k = -1$ and $D = 0$, a contradiction. Therefore $\psi_1$ and $\psi$ induces a nonvanishing map 
$$ \sO_X(D) \to \sO_X(D) \otimes f^*(\sL^k),$$
hence $k = 0$. 
\vskip .2cm \noindent 
(2) If $X$ is a torus or hyperelliptic, then 
$$ h^0(X,T_X \otimes \Omega^1_X) = h^0(X, \Omega_X^1 \otimes \Omega^1_X) = h^0(X,S^2\Omega^1_X) + h^0(X,K_X) \geq 2,$$
hence we conclude again by Proposition \ref{prop:RR2}(4). 
\vskip .2cm \noindent 
(3) Finally, let $p: X \to B $ be a ruled surface over a curve $B$ of genus $g = g(B) \geq 1$. 
Consider the relative tangent bundle sequence
$$ 0 \to T_{X/B} \to T_X \to p^*(T_B) \to 0$$
and its associated extension class $\zeta \in H^1(X,T_{T/B} \otimes p^*(T^*_B)) \simeq H^1(X,-K_X)$. 
Now $H^1(X,-K_X) \ne 0$, unless $ g = 1$ and $X = \mathbb P(\sE)$ with $\sE$ a stable locally free sheaf of rank two on $B$. 
This is a direct consequence of the structure results of  ruled surfaces, \cite[chap. V.2]{Ha77}. 
The latter case ruled out by assumption, 
we can deform the extension class $\zeta$ and obtain a deformation $\sF$ of $T_X $ over $X \times \Delta$. 
Every $\sF_t$ sits in an exact sequence
$$ 0 \to T_{X/B} \to \sF_t \to p^*(T_B) \to p^*(T_B) \to 0.$$ 
Since there are no non-trivial maps $T_{X/B} \to p^*(T_B)$, the sheaves $\sF_t$ are different from $T_X$ for $t \ne 0$,
and we obtain a non-trivial positive-dimensional deformation of $T_X$. \\
It remains to treat the case $X = \mathbb P(\sE)$ with $\sE$ stable over the elliptic curve $B$. 

\end{proof} 

\begin{remark} {\rm Assume that $X = \mathbb P(\sE)$ with $\sE$ a stable locally free sheaf of rank two over an elliptic curve or that 
$X$ is a ball quotient with $H^0(X,S^2\Omega^1_X) \ne 0$. Then $T_X$ has a genuine first order deformation and one might suspect 
that a suitable such deformation is not obstructed. Then Theorem \ref{thm2} could be stated as follows: 
$T_X$ has a genuine non-obstructed deformation if and only if $X$ is neither $\mathbb P_2$ nor a ball quotient with $H^0(X,S^2\Omega^1_X) = 0$. 
}
\end{remark} 

\section{Threefolds} 
\setcounter{lemma}{0}

The Riemann-Roch formula gives, using $c_2(T_X \otimes \Omega^1_X) = -2c_1^2(X) + 6c_2(X)$,  

\begin{proposition}  \label{BL} Let $X$ be a $3$-dimensional compact complex manifold. 
Then $$\chi(X,T_X \otimes \Omega^1_X) = c_1^3(X) - \frac{63}{24}c_1(X) c_2(X).$$
\end{proposition} 

\subsection{Fano threefolds}

\begin{proposition} \label{prop:Fano} Let $X$ be a Fano threefold. Then $T_X$ has a genuine non-obstructed deformation unless $X = \mathbb P_3$. 
\end{proposition} 

\begin{proof} Since $\chi(X,\sO_X) = 1$, we have $c_1(X) c_2(X) = 24$ by Riemann-Roch. Hence 
$$ \chi(X,T_X \otimes \Omega^1_X) = c_1^3(X) - 63.$$ 
By the classification of Fano threefolds, $c_1^3(X) \leq 62$, unless $X = \mathbb P_3$. 
Notice further that 
$$ H^3(X,T_X \otimes \Omega^1_X) = H^0(X, \Omega^1_X \otimes \Omega^2_X) = 0,$$
e.g., since $X$ is rationally connected. 

Hence $$ h^1(X,T_X \otimes \Omega^1_X) - h^2(X,T_X \otimes \Omega^1_X)  > 0,$$
unless $X = \mathbb P_3$. 
Since $q(X) = 0$, this proves the claim. 

\end{proof} 

The arguments actually show more (having in mind that $h^0(X,T_X \otimes \Omega^1_X) \ne 0$)

\begin{corollary} \label{cor2} Let $X$ be a smooth threefold  with $\chi(X,\sO_X) \geq 0$ and 
$$H^0(X,\Omega^1_X) = H^0(X,\Omega^2_X) = H^0(X,\Omega^1_X \otimes \Omega^2_X) = 0.$$ Assume that
$$ c_1^3(X) \leq 63. $$ 
Then  $T_X$ has a genuine non-obstructed deformation.
\end{corollary} 

If $X$ is rationally connected, then the first two conditions in Corollary \ref{cor2} are satisified, hence we obtain

\begin{corollary} \label{cor3} Let $X$ be a smooth rationally connected threefold.  Assume that
$$ c_1^3(X) \leq 63. $$ 
Then  $T_X$ has a non-trivial non-obstructed deformation.
\end{corollary} 

In view of Proposition \ref{prop:Fano} it is natural to consider the case that $X$ is "weak Fano", i.e., $-K_X$ is big and nef. 
In that case, $c_1^3(X) \leq 72$ by \cite[Thm.1.5]{Pro05}. In fact, for a suitable positive integer $m$ the line bundle $-mK_X$ is spanned by global sections and
defines a birational morphism $\varphi: X \to Y$ to a Fano Gorenstein variety $Y$ with at most canonical singularities such that $-K_X = \varphi^*(-K_Y).$ 
By \cite[Thm.1.5]{Pro05}, $(-K_Y)^3 \leq 72$, hence $c_1^3(X) = (-K_X)^3 \leq 72$. The bound $72$ is sharp; actually $(-K_Y)^3 = 72 $ if and only if $Y$ is either 
the weighted projective space $\mathbb P(3,1,1,1)$ or $\mathbb P(6,4,1,1).$ Thus we cannot conclude directly that $T_X$ has a first order or non-obstructed deformation. It should however be
possible to classifiy all $X$ in the range $64 \leq c_1^3(X) \leq 72$ and treat this cases by hand. We give one example, namely 
$$X = \mathbb P(\sO_{\mathbb P_2} \oplus \sO_{\mathbb P_2}(3)).$$
In this case $Y = \mathbb P(3,1,1,1).$ For simplicity, we consider only first order deformations.

\begin{proposition} \label{weak} Let $X$ be the weak Fano threefold $ \mathbb P(\sO_{\mathbb P_2} \oplus \sO_{\mathbb P_2}(3))$.
Then $T_X$ has a genuine first order deformation. 

\end{proposition} 

We prepare the proof by the following

\begin{lemma} \label{lem:A} Let $\pi:X \to S$ be a $\mathbb P_1$-bundle over the smooth compact surface $S$. 
Assume that $h^0(S,T_S \otimes \Omega^1_S) = 1$. 
Then 
$$ h^0(X,T_X \otimes \Omega^1_X) = 1 + h^0(X,T_{X/S} \otimes \Omega^1_X) = 1 + h^0(X,-K_{X/S} \otimes \Omega^1_X).$$
\end{lemma} 

\begin{proof} 
The sequence 
$$ 0 \to T_{X/S} \otimes \Omega^1_X \to T_X \otimes \Omega^1_X \to \pi^*(T_S) \otimes \Omega^1_X \to 0$$
yields in cohomology
$$ 0 \to H^0(X,T_{X/S} \otimes \Omega^1_X) \to H^0(X,T_X \otimes \Omega^1_X) \buildrel{\alpha} \over {\to } H^0(X,\pi^*(T_S) \otimes \Omega^1_X).$$
The sequence 
$$ 0 \to \pi^*(T_S) \otimes \pi^*(\Omega^1_S) \to \pi^*(T_S) \otimes \Omega^1_X \to \pi^*(T_S) \otimes K_{X/S} \to 0 $$
shows
$$ H^0(X,\pi^*(T_S) \otimes \Omega^1_X) = H^0(S,T_S \otimes \Omega^1_S) \simeq \mathbb C.$$
Hence it suffices to observe that $\alpha \ne 0$. This is however clear: $\id: \Omega^1_X \to \Omega^1_X $ yields via $\alpha$ a non-zero morphism $\pi^*(\Omega^1_S) \to \Omega^1_X.$ 
\end{proof} 

\vskip .2cm \noindent
{\it Proof of Proposition \ref{weak} } By Proposition \ref{prop:bir0}, applied to the exceptional section $E := \mathbb P(\sO_{\mathbb P_2}) \simeq \mathbb P_2$ in $X$, it suffices to show
that 
\begin{equation} \label{log}  H^0(X,T_X \otimes \Omega^1_X) = H^0(X,T_X \otimes \Omega^1_X(\log E)). \end{equation} 
We use the exact sequence
$$ 0 \to T_{X/S} \otimes \Omega^1_X(\log E) \to T_X \otimes \Omega^1_X(\log E) \to \pi^*(T_S) \otimes \Omega^1_X(\log E) \to 0. $$ 
Using Lemma \ref{lem:A}, things come down to show 
$$ H^0(X,T_{X/S} \otimes \Omega^1_X) = H^0(X,T_{X/S} \otimes \Omega^1_X(\log E)) $$
and 
$$ h^0(X,\pi^*(T_S) \otimes \Omega^1_X(\log E)) = 1.$$ 
The first equation is seen by taking cohomology of the exact sequence
$$ 0 \to T_{X/S} \otimes \Omega^1_X \to T_{X/S} \otimes \Omega^1_X(\log E) \to {T_{X/S}}{\vert E} \to 0$$
and by observing 
$$ H^0(E, {T_{X/S}}{\vert E}) = H^0(E,{-K_{X/S}}{\vert E}) = H^0(E,N_{E/X}) = 0.$$
The second equation follows from the observation 
$$ \pi_*(\Omega^1_X) = \pi_*(\Omega^1_X(\log E)),$$ which is seen
either by restricting to the fibers of $\pi$ or by noticing that, taking $\pi_*$,  the induced morphism $\sO_S \to R^1\pi_*(\Omega^1_X) $ is injective. 
Thus Equation (\ref{log}) is shown and the proof of  Proposition \ref{weak} is complete. 

\vskip .2cm
In Subsection  \ref{sect:P1} we come back to $\mathbb P_1$-bundles  over surfaces in general. How the later results do not yield Proposition \ref{weak}.

\subsection{$\mathbb P_2$-bundles} 

We start to study threefolds carrying a projective bundle structure by studying $\mathbb P_2$-bundles. 

\begin{theorem} \label{P1} Let $\pi:X \to C$ be a $\mathbb P_2$-bundle over the smooth projective curve $C$. Then $T_X$ has a genuine first order deformation.
\end{theorem} 

\begin{proof} Write $X = \mathbb P(\sF)$ with a locally free sheaf $\sF$ of rank three on $C$. Let $g$ be the genus of $C$. 
If $g = 0$, then $(-K_X)^3 = 54 $, hence we conclude by Corollary \ref{cor3}. 
To compute $(-K_X)^3$, just use the formula
$$ -K_X = \sO_{\mathbb P(\sF)}(3) \otimes \pi^*(\det \sF^*) \otimes \sO_C(-K_{C})),$$
see e.g. \cite[Ex.III.8.4]{Ha77}.

Thus we will assume from now on that  $g \geq 1$. In this case 
$\chi(X,\sEnd_0(T_X)) \geq 0$, so a more detailled investigation has to be made. 
Taking cohomology of the exact sequence
$$ 0 \to T_{X/C} \otimes \Omega^1_X \to T_X \otimes \Omega^1_X \to \pi^*(T_C ) \otimes \Omega^1_X \to 0 $$
and observing that the morphism
$$ H^0(X,T_X \otimes \Omega^1_X) \to H^0(X,\pi^*(T_C) \otimes \Omega^1_X) = H^0(X,\pi^*(T_C \otimes \Omega^1_C)) \simeq \mathbb C$$
is surjective, leads to an exact sequence
$$ 0 \to H^1(X,T_{X/C} \otimes \Omega^1_X) \to H^1(X,T_X \otimes \Omega^1_X) \to H^1(X,\pi^*(T_C) \otimes \Omega^1_X) \to $$
$$ \to  H^2(X,T_{X/C} \otimes \Omega^1_X) .$$ 
We will now show 
\begin{enumerate}
\item $ h^1(X,T_{X/C} \otimes \Omega^1_X) \geq g$;
\item $H^1(X,\pi^*(T_C) \otimes \Omega^1_X) \ne 0$;
\item $ H^2(X,T_{X/C} \otimes \Omega^1_X) = 0$.
\end{enumerate}
Then the exact sequence yields 
$$ h^1(X,T_X \otimes \Omega^1_X) > g = h^1(X,\sO_X),$$
which was to be proved. \\
{\it Proof of (1).} Write $\sL := \sO_{\mathbb P(\sF)}(1)$ and tensor the relative Euler sequence 
$$ 0 \to \sO_X \to \pi^*(\sF^*) \otimes \sL \to T_{X/C} \to 0 $$
by $\Omega^1_X$ to obtain an exact sequence 
$$ H^1(X,T_{X/C} \otimes \Omega^1_X) \to H^2(X,\Omega^1_X) \to H^2(X,\pi^*(\sF^*) \otimes \sL \otimes \Omega^1_X).$$
Now  $H^2(X,\pi^*(\sF^*) \otimes \sL \otimes \Omega^1_X) = 0$ via the Leray spectral sequence. Further $h^2(X,\Omega^1_X) = g > 0$:  
use Hodge decomposition and  $H^{3,0} = 0$ to obtain $$ h^2(X,\Omega^1_X) =  \frac {b_3(X)}{2} = \frac{b_1(C)}{2} = g(C).$$ 
Hence (1) follows. \\
{\it Proof of (2).} Since $\pi_*(\Omega^1_X) = \Omega^1_C$, we have
$$  h^1(X,\pi^*(T_C) \otimes \Omega^1_X) \geq h^1(C, T_C \otimes \Omega^1_C) = g, $$
proving (2). \\
{\it Proof of (3).} This follows again by the Leray spectral sequence.

\end{proof} 

\subsection{$\mathbb P_1$-bundles}  \label{sect:P1}

In many cases the non-rigidity of the tangent bundle of a $\mathbb P_1$-bundle over a surface $S$ can be established as follows. For simplicity, we assume that 
$q(X) = H^1(X,\sO_X) = H^1(S,\sO_S) = 0$.

\begin{proposition} \label{prop:bundle}  Let $\pi: X \to S$ be a $\mathbb P_1-$bundle over the smooth compact surface $S$ with $q(S) = 0$. 
If
$$ H^1(X, T_X \otimes \pi^*(\Omega^1_S)) = H^1(S, \pi_*(T_X) \otimes  \Omega^1_S) \ne 0,$$
then $T_X$ has a genuine first order deformation. 

\end{proposition} 

\begin{proof}  We use the exact sequence 
\begin{equation} \label{exseq1}  0 \to  T_X \otimes \pi^*(\Omega^1_S) \to T_X \otimes \Omega^1_X \to T_X \otimes K_{X/S} \to 0. \end{equation} 
From the exact sequence
\begin{equation} \label{exseq2}  0 \to \sO_X = -K_{X/S} \otimes K_{X/S} \to T_X \otimes K_{X/S} \to \pi^*(T_S) \otimes K_{X/S} \to 0, \end{equation} 
we deduce $h^0(X,T_X \otimes K_{X/S}) = 1.$ 
Now 
$$ H^0(X,T_X \otimes \Omega^1_X) \to H^0(X,T_X \otimes K_{X/S}) $$
does not vanish: $\id: T_X \to T_X$ induces a non-zero morphism $-K_{X/S} \to T_X$. 
Hence by Sequence (\ref{exseq1}),  $$H^1(X,T_X \otimes \pi^*(\Omega^1_S) = H^1(S,\pi_*(T_X) \otimes \Omega^1_S) $$ injects into $H^1(X,T_X \otimes \Omega^1_X )$. \\
\end{proof} 

\begin{corollary}  Let $\pi: X \to S$ be a $\mathbb P_1-$bundle over the smooth compact surface $S$ with $q(S) = 0$. 
If $$ h^0(S,T_S \otimes \Omega^1_S) = 1$$  and if 
$$h^1(X, T_{X/S} \otimes \pi^*(\Omega^1_S)) \geq 2,$$
then $T_X$ has a genuine first order deformation. 
\end{corollary} 

\begin{proof} This is immediate, taking cohomology of
$$ 0 \to T_{X/S} \otimes \pi^*(\Omega^1_S) \to T_X \otimes \pi^*(\Omega^1_S) \to \pi^*(T_S \otimes \Omega^1_S) \to 0.$$
%In fact, the map 
%$$ H^0(X, T_X \otimes \pi^*(\Omega^1_S)) \to H^0(X,\pi^*(T_S \otimes \Omega^1_S)) \simeq \mathbb C$$
%is surjective, since $\id: T_S \to T_S$ lifts to $\id: T_X \otimes T_X$. 

\end{proof}

With a little more care and a slighty stronger assumption on $h^1(X, T_{X/S} \otimes \pi^*(\Omega^1_S))$, but without assumption on $q(S)$,  we obtain non-obstructed deformations: 

\begin{proposition} Let $\pi: X\to S$ be a $\mathbb P_1$-bundle. Assume that 
$$ h^0(S,T_S \otimes \Omega^1_S) = 1$$ and that 
$$ h^1(X,T_{X/S} \otimes \pi^*(\Omega^1_S)) \geq 3.$$
Then $T_X$ has a genuine non-obstructed deformation.
\end{proposition}

\begin{proof} We consider the tangent bundle sequence 
$$ 0 \to T_{X/S} \lra T_X {\buildrel{\alpha} \over { \lra}} \pi^*(T_S) \to 0.$$ 
Since
$$ \dim {\rm Ext}^1(\pi^*(T_S),T_{X/S}) = h^1(X,T_{X/S} \otimes \pi^*(\Omega^1_S)) \geq 3,$$
we obtain an at least three-dimensional family of extensions
$$ 0 \to T_{X/S} \buildrel {\beta} \over {\lra} \sE_t \lra \pi^*(T_S) \to 0,$$ 
and it suffices to show that $\sE_t \not \simeq T_X$ for general $t$. 
Assume to the contrary that $\sE_t \simeq T_X$. 
Then the composed map $\alpha  \circ \beta: T_{X/S} \to \pi^*(T_S) $ must vanish (restrict to fibers of $\pi$). 
Hence $\beta$ induces a morphism $T_{X/S} \to T_{X/S} $ which must be a multiple of the identity map. 
Thus we have an induced map $\pi^*(T_S) \to \pi^*(T_S)$, which by assumption is another multiple of the identity.
Hence the space of extension is twodimensional, contradicting our dimension assumption. 

\end{proof} 

We now give a criterion for the nonvanishing of $ H^1(X,T_{X/S} \otimes \pi^*(\Omega^1_S))$. 
We assume for simplicity that the $\mathbb P_1$-bundle $X \to S$ is actually of the form $X = \mathbb P(\sF)$ with a locally free sheaf $\sF$ of rank two on $S$;
this can always arranged by passing to a finite \'etale cover of $S$, see \cite{El82}.

\begin{corollary} Let $\sF$ be a locally free sheaf of rank two on the smooth compact complex surface $S$ with $q(S) = 0$ and set $X = \mathbb P (\sF)$.
Assume that $$h^0(S,T_S \otimes \Omega^1_S) = 1$$ and that 
$$ 2c_1^2(\sF) - 8c_2(\sF) + \frac{1}{2}( K_S^2 - 5 c_2(S)) \leq -2.$$
Then $T_X$ has a genuine first order deformation.
This happens e.g., when $\sF$ is $\omega$-semistable for some K\"ahler form $\omega$, so that $c_1^2(\sF) \leq 4c_2(\sF)$.
\end{corollary} 

\begin{proof} Notice that 
$$ \pi_*(T_{X/S}) = S^2(\sF) \otimes \det \sF^*, $$ thus
 $$h^1(S,S^2(\sF) \otimes \det \sF^* \otimes \Omega^1_S) = h^1(X,T_{X/S} \otimes \pi^*(\Omega^1_S).$$
 Hence it suffices to show that 
 $$ \chi(S,S^2(\sF) \otimes \det \sF^* \otimes \Omega^1_S) \leq -2,$$
 which is equivalent by Riemann-Roch to our assumption. \\
 Notice finally that since $q(S) = 0$, then $K_S^2 \leq 3c_2(S)$ and $c_2(S) = \chi_{\rm top}(S) \geq 3$. 
 \end{proof} 
 
 In the same manner, we have

\begin{corollary} Let $\sF$ be a locally free sheaf of rank two on the smooth compact complex surface $S$  and set $X = \mathbb P (\sF)$.
Assume that $h^0(S,T_S \otimes \Omega^1_S) = 1$ and that 
$$ 2c_1^2(\sF) - 8c_2(\sF) + \frac{1}{2}( K_S^2 - 5 c_2(S)) \leq -3.$$
Then $T_X$ has a genuine unobstructed deformation.
\end{corollary}

\begin{remark} The equation 
$H^1(X,T_X \otimes \pi^*(\Omega^1_S)) \ne 0 $ also holds under the assumptions 
$$H^1(S, T_S \otimes \Omega^1_S) \ne 0 $$
 and $$h^0(S,S^2(\sF^*) \otimes \det \sF \otimes \Omega^1_S) = 0,$$ 
again by taking cohomology of the exact sequence 
$$ 0 \to T_{X/S} \otimes \pi^*(\Omega^1)S) \to T_X \otimes \pi^*(\Omega^1_S) \to \pi^*(T_S \otimes \Omega^1_S) \to 0.$$

\end{remark}

\begin{remark} In our setting $X = \mathbb P(\sF)$, we have
$$ \chi(X,\sEnd_0(T_X)) = 2c_1^2(\sF) - 8 c_2(\sF) + \frac{2}{3} \big(K_S^2 - 8 c_2(S) \big).$$
So if this number is negative and if 
$$H^3(X,\sEnd_0(T_X))) = 0,$$ 
then $H^1(X,\sEnd_0(T_X)) \ne 0$. 
The last vanishing amounts via Serre duality to 
$$ h^0(X,\Omega^1_X \otimes \Omega^2_X) \leq h^0(X,K_X)).$$ 
\end{remark}

\subsection{Blow-ups} 

We conclude the section by studying blow-ups. First we recollect our knowledge on blow-ups of points. 

\begin{proposition} \label{blow1} Let $Y$ be a smooth compact complex threefold and $\pi: X\to Y$ be the blow-up at $y_0$ with exceptional  divisor $E$. 
Then $T_X$ has a genuine first order deformation provided one of the following holds. 
\begin{enumerate} 
\item $Y$ is rationally connected with $c_1^3(Y) \leq 71$;
\item $Y$ is smooth threefold with $c_1^3(Y) \leq 71$ and $Y$ is uniruled. \\
If the MRC fibration has two-dimensional (smooth) image $S$, suppose
further that $\chi(S,\sO_S) \ne 0$ and that  $H^0(S,\Omega^1_S \otimes K_S) = 0$;
\item  $T_Y$ has a genuine first order deformation and  $h^0(Y,T_Y \otimes \Omega^1_Y) < 9$;
\item $h^0(Y,T_Y \otimes \Omega^1_Y) = 1$, i.e., $T_Y$ is simple.
\end{enumerate} 
In the cases (1) and (2) $T_X$ has even a genuine non-obstructed deformation; the same being true in case (3) provided $T_Y$ has a non-obstructed genuine deformation. 

\end{proposition} 

\begin{proof} Assertion (1) is settled by Corollary \ref{cor3}; (3) by Corollary \ref{cor:blow-up} and (4) by Theorem \ref{thm:gendiv}. Thus
only (2) needs to be proven.
We calculate
$$ \chi(X,\sEnd_0(T_X)) =  \chi(X,T_X \otimes \Omega^1_X) - \chi(X,\sO_X)  = $$
$$ = c_1^3(X) - 64 \chi(X,\sO_X) = c_1^3(X)  - 64\chi(S,\sO_S) \leq 63 - 64\chi(S,\sO_S).$$ 
Hence by our assumption $\chi(S,\sO_S) \geq 1$ and therefore $ \chi(X,\sEnd_0(T_X)) <  0$ and therefore
$$ h^1(X,\sEnd_0(T_X)) - h^2(X,\sEnd_0(T_X))  + h^3(X,\sEnd_0(T_X) ) > 0. $$
Hence it suffices to show that 
\begin{equation} \label{eq:H3} H^3(X,T_X \otimes \Omega^1_X) = 0. \end{equation} 
By Serre duality, 
$$ H^3(X,T_X \otimes \Omega^1_X) = H^0(X,T_X \otimes \Omega^1_X \otimes \sO_X(K_X)) = H^0(X,\Omega^1_X \otimes \Omega^2_X).$$
We may assume that $Y$ is not rationally connected. $Y$ being uniruled, we consider the MRC fibration $f: X \dasharrow S$, \cite{Kol96},  with $S$ not uniruled and smooth. 
If $\dim S = 2$, Proposition \ref{prop:forms} shows that 
$$ H^0(X,\Omega^1_X \otimes \Omega^2_X) = H^0(S, \Omega^1_S \otimes K_S).$$
If $\dim S = 1$, then $S$ is a smooth curve of genus at least two, and $f$ is actually a morphism. 
But then, restricting to a general fiber $F$, a rational surface, a direct calculation shows that 
$$ H^0(X,\Omega^1_X \otimes \Omega^2_X) = 0,$$
proving (\ref{eq:H3}).

\end{proof} 

In the proof of the next proposition we will use the following

\begin{notation} Let $X$ be a normal complex algebraic variety. Then 
$$ \Omega^{[q]}_X := (\bigwedge^q \Omega^1_X)^{**}$$
denotes the sheaf of reflexive $q$-forms. 
If $X$ has canonical singularites and if $\pi: \hat X \to X$ is a desingularization, then by \cite{GKKP11}, 
$$  \Omega^{[q]}_X = \pi_*(\Omega^q_{\hat X}).$$ 
\end{notation}

\begin{proposition} \label{prop:forms} Let $X$ be a smooth projective threefold. Suppose that $X$ is uniruled with MRC fibration $f: X \dasharrow S$ to the smooth projective
surface $S$. Then the pull-back
$$  f^*:  H^0(S, \Omega^1_S \otimes K_S) \to H^0(X,\Omega^1_X \otimes \Omega^2_X)$$
is an isomorphism. 
\end{proposition} 

\begin{proof}
Note that $\dim h^0(X,\Omega^1_X \otimes \Omega^2_X) $ is a birational invariant of smooth projective manifolds. Hence we may assume that $f$ is a morphism. 
Clearly, $f^*$ is injective. Running a relative MMP, we obtain a factorization
$$ X \dasharrow X' \to S,$$
where $X \dasharrow X'$ is a sequence of relative contractions and flips, and where $f': X' \to S$ is a Mori fiber space. \\
We proceed with the following observations. 
If $Y$ and $Z$ are normal projective varieties with terminal singularities and if $\varphi: Y \to Z$ is a divisorial contraction, then 
$$\varphi_*((\Omega^1_Y \otimes \Omega^2_Y)^{**}) \subset (\Omega^1_Z \otimes \Omega^2_Z)^{**}.$$ 
Moreover, if $\varphi: Y \dasharrow Z$ is a flip, then 
$$ H^0(Y,(\Omega^1_Y \otimes \Omega^2_Y)^{**}) = H^0(Z, (\Omega^1_Z \otimes \Omega^2_Z)^{**}).$$
Hence it suffices to consider $f': X' \to S$, which is equidimensional and a conic bundle outside a finite set of $S$, i.e., there is a finite set $A \subset S$ such that $f': X' \setminus (f')^{-1}(A) \to S \setminus A$ 
is a conic bundle, 
and we need to prove that 
$$ (f')^*: H^0(S, \Omega^1_S \otimes \Omega^2_S) \to H^0(X', (\Omega^1_{X'} \otimes \Omega^2_{X'})^{**}) $$
is surjective. 
Since $f'$ is a submersion outside a set of dimension at most one,  $\Omega^1_{X'/S}$ is torsion free. We will use the exact sequences
$$ 0 \to f'^*(\Omega^1_S) \to \Omega^1_{X'} \to \Omega^1_{X'/S} \to 0$$
and 
$$ 0 \to f'^*(\Omega^2_S) \to \Omega^{[2]}_{X'} \to \sS \to 0, $$
where 
 $\sS$ is a torsion free sheaf with $\sS \vert X'_{0} = (f'^*(\Omega^1_S) \otimes \Omega^1_{X'/S}) \vert X'_{0}$
and where $X'_0$ is the regular locus of $X'$ and $f'$.
Now we observe that 
$$  H^0(X',  \big( \Omega^1_{X'/S} \otimes \Omega^{[2]}_{X'} \big) ^{**}) = 0 $$
and 
$$ H^0(X', \sS \otimes (f')^*(\Omega^1_S) ) = 0.$$
Indeed, both sheaves in question are negative on the general fiber of $f'$. 
Then the assertion follows, tensoring the first exact sequence with $\Omega^{[2]}_{X'}$, and then the second with $(f')^*(\Omega^1_S)$ and computing on $X'_0$. 

\end{proof}

Proposition \ref{blow1} suggests to proceed by induction on the Picard number $\rho(X)$, performing an MMP. In this context, we notice that 
in a similar way as in Proposition \ref{blow1},  it possible to treat the other contraction of extremal rays on threefolds
which contract a divisor $E$ to a point. This opens a way to reduce the problem to Mori fiber spaces and threefolds with nef canonical bundles (possibly singular). 

%Of course, this is merely a special case of Theorem \ref{thm:gendiv} in case $T_X$ has no non-trivial endomorphisms, and 
%the process can iterated.

Blowing up curves is more complicated; we restrict ourselves to first order deformations.

\begin{theorem} \label{june20} Let $Y$ be a smooth projective threefold such that $$h^0(Y,T_Y \otimes \Omega^1_Y) = 1.$$ Let $\pi: X \to Y$ be the blow-up of a 
smooth curve $C \subset Y$. Assume that 
\begin{equation} \label{g} -K_Y \cdot C \geq 2. \end{equation} 
Then $H^1(X,\sEnd_0(T_X)) \ne 0$. 

\end{theorem} 

\begin{proof} It suffices to show that 
$$ H^1(Y,\pi_*(\sEnd_0(T_X)) \ne 0.$$ 
To prove this, we consider the canonical exact sequence
$$ 0 \to \pi_*(\sEnd_0(T_X)) \to \sEnd_0(T_Y) \to Q \to 0, $$
where $Q$ is a coherent sheaf supported on $C$. 
Then things comes down to show that 
\begin{equation} \label{march9} H^0(Y,Q) \ne 0. \end{equation} 
The sheaf $Q$ appears as well in the exact sequence 
$$ 0 \to \pi_*(T_X \otimes \Omega^1_X) \to T_{Y} \otimes \Omega^1_{ Y} \to Q \to 0,$$
and we shall work with this sequence. 
Set $E = \pi^{-1}(C)$. The normal bundle of $C $ in $Y$ will simply be denoted $N_C$. 
Using the exact sequences
$$ 0 \to \pi_*(T_X) \otimes \Omega^1_{Y} \to \pi_*(T_X \otimes \Omega^1_X) \to \pi_*(T_X \vert E \otimes \Omega^1_{E/C}) \to 0 $$
and
$$ 0 \to \pi_*(T_X) \otimes \Omega^1_{Y} \to  T_{Y} \otimes \Omega^1_{Y} \to N_C \otimes \Omega^1_{Y} \to 0, $$
a diagram chase yields an exact sequence
$$ 0 \to \pi_*(T_X \vert E \otimes \Omega^1_{E/C} ) \to N_C \otimes \Omega^1_{Y} \vert C \to Q \to 0. $$ 
Since $$ h^0(C, \pi_*(T_X \vert E \otimes \Omega^1_{E/C}) ) = h^0(E, T_X \vert E \otimes \Omega^1_{E/C} ) = $$
$$ = h^0(E,T_E \otimes \Omega^1_{E/C}) = h^0(E, T_{E/C} \otimes \Omega^1_{E/C}) = h^0(E,\sO_E) = 1,$$
Equation (\ref{march9}) comes therefore down to show that
\begin{equation} \label{blow4}  h^0(C,N_C \otimes \Omega^1_{Y} \vert C) \geq 2.\end{equation} 
This follows the stronger inequality
\begin{equation} \label{blow5} \chi(C,N_C \otimes \Omega^1_{Y} \vert C) \geq 2,\end{equation}
which by Riemann-Roch is equivalent to 
$$ 6(1-g) + c_1(N_C \otimes {\Omega^1_Y} \vert C) \geq 2.$$
This is just our assumption via the adjunction formula. 
\end{proof}

\begin{corollary}  Let $Y$ be a Fano threefold such that $h^0(Y,T_Y \otimes \Omega^1_Y) = 1.$ Let $\pi: X \to Y$ be the blow-up of a 
smooth curve $C \subset Y$. Then $$H^1(X,\sEnd_0(T_X)) \ne 0.$$
\end{corollary} 

\begin{proof} It remains to treat the case that $-K_Y \cdot C = 1$. But then $C$ is a smooth rational curve, and Theorem \ref{thm:gendiv} applies. 
In fact, if $-K_Y \cdot C = 1$, then $Y$ must have index one. In almost all cases, $-K_Y$ is spanned, hence defines a morphism $\varphi: Y \to \mathbb P_N$
such that $\varphi^*(\sO_{\mathbb P_N})(1) = \sO_Y(-K_Y)$. Hence $\varphi(C)$ must be a line $\ell$ and $C \to \ell$ is an isomorphism. 
There are only two exceptional cases, \cite[p.49]{IP99} which can be checked by hand. 

\end{proof}

Keeping track of the $H^1$-term in $\chi(N_C \otimes \Omega^1_Y \vert C)$ and using Serre duality, the proof of Theorem \ref{june20} actually shows the slightly 
stronger 

\begin{corollary} Let $Y$ be a smooth projective threefold such that $$h^0(Y,T_Y \otimes \Omega^1_Y) = 1.$$ Let $\pi: X \to Y$ be the blow-up of a 
smooth curve $C \subset Y$. If 
$$ -K_Y \cdot C + h^0(N_C \otimes \Omega^2_Y \vert C) \geq 2,$$ 
then $H^1(X,\sEnd_0(T_X)) \ne 0$. 
\end{corollary} 

\begin{proof} It suffices to note that in case $ g = 0$, Theorem \ref{thm:gendiv} applies. 
\end{proof} 

\section{Calabi-Yau threefolds: birational morphisms and flops}

\setcounter{lemma}{0}

We are now turning to Calabi-Yau manifolds, mostly in dimension three. To be precise, a Calabi-Yau manifold is a {\it  simply connected projective manifold $X$ with $K_X \simeq \sO_X$.}  Therefore $T_X$ has a genuine first
order deformation if and only if $H^1(X,T_X \otimes \Omega^1_X) \ne 0$.

\begin{definition} Let $X$ be a Calabi-Yau manifold and $\varphi: X \to Y$ be a birational morphism
to a normal projective variety $Y$. Then $\varphi$ is said to be primitive if the relative Picard number $\rho(X/Y) = 1.$ 
A primitive contraction is divisorial if the exceptional locus $E$ of $\varphi$ has codimension $1$. Then automatically $E$ is
an irreducible divisor. 
\end{definition} 

We first collect a few known results on divisorial contractions in dimension three.
\begin{proposition} \label{prop:contraction} Let $X$ be a Calabi-Yau threefold, $\varphi: X \to Y$ a primitive divisorial birational map 
contracting the irreducible 
divisor $E$. Then the following holds. 
\begin{enumerate} 
\item $K_Y \simeq \sO_Y$ and $Y$ has canonical singularities; we say that $Y$ is a weak Calabi-Yau variety.
\item If $\dim \varphi(E) = 0,$ then $-K_E$ is ample, so $E$ is a (possibly singular) del Pezzo surface.
\item If $\dim \varphi(E) = 1,$ then $C := \varphi(E)$ is a smooth curve. The map $\varphi_{\vert E}$ defines a conic bundle structure on $E$. 
\end{enumerate} 
\end{proposition} 

\begin{proof} We refer to the fundamental papers of Wilson \cite{Wi92}, \cite{Wi93}, \cite{Wil94}, \cite{Wi97}, \cite{Wil99}. 
\end{proof} 

A primitive contraction might also be a small birational contraction, in which case the exceptional locus is a finite union of smooth rational curves. This case will be treated at the end
of this section. Or we have $\dim Y= 1$ or $2$; then $\varphi$ is a K3-fibration, an abelian fibration or an elliptic fibration. These cases will be treated in a different paper. 
Note also that it is expected that a Calabi-Yau threefold $X$ with $\rho(X) \geq 2$ should always admit a (non-trivial) contraction. This follows from two standard conjectures, namely
that the Mori cone $\overline{NE}(X)$ is locally rational polyhedral and that nef line bundles on Calabi-Yau manifolds should be semiample; see e.g. the above cited papers of Wilson or
the survey \cite{LOP18}. 

Recall that, given be a normal complex algebraic variety $X$,  then 
$$ \Omega^{[q]}_X := (\bigwedge^q \Omega^1_X)^{**}$$
is the sheaf of reflexive $q$-forms. Further, 
if $X$ has canonical singularites and if $\pi: \hat X \to X$ is a desingularization, then by \cite{GKKP11}, 
$$  \Omega^{[q]}_X = \pi_*(\Omega^q_{\hat X}).$$

\begin{lemma} \label{lem:stability}
Let $X$ be a Calabi-Yau manifold, $\varphi: X \to Y$ a primitive birational morphism to a normal projective variety, contracting an irreducible divisor $E$. 
Then the following holds. 
\begin{enumerate} 
\item The reflexive cotangent sheaf $\Omega^{[1]}_Y $ is $H$-polystable for any ample divisor $H$ on $X.$ 
\item If $\Omega^{[1]}_Y $ is not $H$-stable for some ample divisor $H$, then there exists a quasi-\'etale cover $\eta: Z \to Y$ such that 
$Z$ decomposes into a product an abelian variety and (possibly singular) Calabi-Yau varieties and irreducible symplectic varieties (in the sense of \cite{GKP16})
(possibly $Z$ is just an abelian variety or there is no abelian factor).
\item Assume $n = 3$ and that $\Omega^{[1]}_Y $ is not $H$-stable for some ample divisor $H$. Then $\dim \varphi(E) = 1$ and there exists a
quasi-\'etale cover $\eta: Z \to Y$ such that either $Z$ is an abelian threefold or $Z = B \times S$ with $Z$ an elliptic curve and $S$ a K3-surface (possibly with rational
double points). 
\end{enumerate} 
\end{lemma}

\begin{proof} Assertion (1) has been shown by Guenancia \cite{Gu16}, while Assertion (2) is the main result in \cite{HP19}. Note that either 
$Z$ is abelian or $Z$ has a product decomposition with at least two factor, since Calabi-Yau or irreducible symplectic varieties have stable 
tangent sheaves (even after quasi-\'etale cover). \\
So it remains to prove (3). By (2), it is clear that $Z$ is either abelian or a product $B \times S$ with $B$ an elliptic curve and $S$ a K3-surface. 
So suppose that $\dim \varphi(E) = 0$, i.e., $Y$ has a single singular points $y_0$, a case which we will rule out. 
Note that $Z$ can have at most finitely many singular point, hence if $ Z = C \times S$, 
hence $S$ is smooth. Thus $Z$ always smooth. We consider the orbifold Euler characteristic $e_{\rm orb}(Y)$, see e.g. \cite[2.14]{Bl96}. By \cite[Cor. p.26]{Bl96}, 
$$ e_{\rm orb}(Z)  = \deg(\eta) \cdot e_{\rm orb}(Y).$$
Now  $ e_{\rm orb}(Z) = e_{\rm top}(Z) = 0$.
On the other hand, 
$$ e_{\rm orb}(Y) = e_{\rm top}(Y) - (1 - \frac{1}{{\rm ord G_{y_0}}}),$$
where $G_{y_0}$ is the group attached to the quotient singularity $y_0$,
which is absurd since $e_{\rm top}(Y)$ is an integer. 
\end{proof}

The case $n = 3$ and $E$ is contracted to a curve is more complicated; here we simply state what is needed later. 

\begin{lemma} \label{lem:stability2}
Let $X$ be a Calabi-Yau threefold, $\varphi: X \to Y$ a primitive birational morphism to a normal projective variety, contracting an irreducible divisor $E$
to the smooth curve $C$. If $C \simeq \mathbb P_1$ (or, more generally, the genus $g(C)$ is even) and if $E$ is a $\mathbb P_1$-bundle over $C$, then $\Omega^{[1]}_Y$ is $H$-stable for any ample divisor $H$.
\end{lemma} 

\begin{proof} We argue as in the previous proof, part (3). Still we have $e_{\rm orb}(C \times S) = 0$, even if $S$ is singular. This is immediate from
the definition of the orbifold Euler number, since $C$ is a elliptic curve. 
Since $E$ is a $\mathbb P_1$-bundle, $Y$ is locally of the form $\Delta \times (A_1)$ with $\Delta$ a small disc in $\mathbb C$. Thus for all $y \in C$, 
the group $G_y$ has order $2$ and by definition of the orbifold Euler number,
$$ e_{\rm orb}(Y) = e_{\rm top}(Y) - (1 - \frac{1}{2} e_{\rm top}(C)).$$
Since $C \simeq \mathbb P_1$ by assumption, we conclude that 
$$ e_{\rm top}(Y) = 1,$$ 
hence $b_3(Y) $ is odd. 
This is impossible: either use the Leray spectral sequence to deduce that $H^3(X,\mathbb C) = H^3(Y,\mathbb C)$. But $b_3(X)$ is even due to Hodge decomposition on $X$. 

\end{proof} 

\begin{remark} \label{rem:stability} If $C$ is an elliptic curve in the setting of Lemma \ref{lem:stability2}, then we conclude that $\chi_{\rm top}(Y) = 0$, hence
$\chi_{\rm top}(X) = 4$.
\end{remark}

\begin{proposition} \label{prop:contract} Let $X$ be a Calabi-Yau manifold, $\varphi: X \to Y$ a primitive divisorial birational map 
contracting the irreducible divisor $E$. Suppose that $\Omega^{[1]}_Y$ is $H$-stable for some ample divisor $H$. 
Then 
$$ \dim H^0(X,\Omega^1_X \otimes T_X) = \dim H^0(X, \sO_X(E) \otimes \Omega^1_X \otimes T_X ) = 1.$$
\end{proposition} 

\begin{proof} The equation
$ \dim H^0(X,\Omega^1_X \otimes T_X)  = 1$ follows from the stability of $T_X.$ 
To obtain the second equation, observe that
any non-zero section of $\Omega^1_X(E) \otimes T_X$ can be seen as a non-zero morphism 
$$ \lambda: \Omega^1_X(-E) \to \Omega^1_X.$$ 
Taking direct images, this gives a morphism
$$ \varphi_*(\lambda):  \varphi_*(\Omega^1_X(-E)) \to \varphi_*(\Omega^1_X). $$
Since $\varphi_*(\Omega^1_X(-E)) \subset \Omega^{[1]}_Y,$ with equality outside $S = \varphi(E)$, the map 
$\varphi_*(\lambda) $ gives a morphism
$$ \mu': (\Omega^{[1]}_Y) _{ \vert Y  \setminus S} \to  (\Omega^{[1]}_Y) _{ \vert Y \setminus S}. $$
Since $S$ has codimension at least $2$ in $Y,$ and since $\Omega^{[1]}_Y$ is reflexive, $\mu'$ extends to a morphism
$$ \mu: \Omega^{[1]}_Y \to  \Omega^{[1]}_Y. $$
By construction, $\varphi_*(\lambda) = \mu \circ \iota$, where 
$$\iota: \varphi_*(\Omega^1_X(-E)) \to \Omega^{[1]}_Y $$
is the inclusion map.
Now $\Omega^{[1]}_Y$ is stable for any ample line bundle on $Y$ by Lemma \ref{lem:stability}. In particular, $\Omega^{[1]}_Y$ is simple, thus 
there exists a complex number $c \ne 0$ such that 
$$ \mu = c  \ {\rm id}.$$ 
It follows that 
$\lambda _{ \vert X \setminus E } = c \ {\rm id}_{\Omega^1_{X \setminus E}},$ 
and therefore $\lambda = c  \ {\rm id}_{\Omega^1_X} \circ \kappa,$
where $\kappa: \Omega^1_X(-E) \to  \Omega^1_X$ is the inclusion map. This proves the assertion. 

\end{proof}  

\begin{lemma} \label{jun28} With $n = \dim X$, suppose in the setting of Proposition \ref{prop:contract} additionally that $H^{n-1}(E,\Omega^1_X \otimes T_X \vert E) \ne 0$.
Then
$$ H^1(X,T_X \otimes \Omega^1_X) \ne 0.$$

\end{lemma}

\begin{proof}  We show equivalently that $H^{n-1}(X,\Omega^1_X \otimes T_X) \ne 0. $
Consider the cohomology sequence
$$ H^{n-1}(X,\Omega^1_X \otimes T_X) \to H^{n-1}(E, \Omega^1_X \otimes T_X \vert E) \to $$
$$ \to H^n(X, \mathcal O_X(-E) \otimes \Omega^1_X \otimes T_X)
\to H^n(X,\Omega^1_X \otimes T_X) \to 0.$$ 
From Proposition \ref{prop:contract} and Serre duality, we  know that 
$$ \dim H^n(X, \mathcal O_X(-E) \otimes \Omega^1_X \otimes T_X)
= \dim  H^n(X,\Omega^1_X \otimes T_X).$$ This yields the claim.

\end{proof}

\begin{theorem} \label{thm:cont1}  Let $X$ be a Calabi-Yau manifold, $\varphi: X \to Y$ a primitive divisorial birational map 
contracting the irreducible 
divisor $E$. If $H^0(E,T_E) \ne 0$, and if $\Omega^{[1]}_Y$ is $H$-stable for some ample divisor $H$, then
$$ H^1(X, \Omega^1_X \otimes T_X) \ne 0.$$
\end{theorem} 

\begin{proof}  We aim to apply Proposition \ref{jun28} and verify that 
$$ H^{n-1}(E, \Omega^1_X \otimes T_X) \vert E) = H^0(E, T_X \otimes \Omega^1_X \otimes K_E) \ne 0.$$ 
In fact, $T_X \otimes \Omega^1_X \vert E \otimes K_E$ contains - via the (co)tangent sequence and the adjunction formula -  the subsheaf $$T_E \otimes N^*_E \otimes K_E \simeq  T_E,$$
hence the nonvanishing follows from our assumption $H^0(E,T_E) \ne 0$. 

\end{proof} 

%\begin{proof} By Corollary \ref{prop:bir1},  it suffices to show that 
%\begin{equation}  \label{eq1} \dim H^0(X,\Omega^1_X(\log E) \otimes T_X) = 1.\end{equation} 
%By Proposition \ref{prop:contract}, 
%\begin{equation}  \label{eq2} \dim H^0(X,\Omega^1_X(E) \otimes T_X) = 1.\end{equation} 
%Since $\Omega_X^1(\log E) \subset \Omega^1_X \otimes \sO_X(E),$
%equation(\ref{eq1}) follows from equation(\ref{eq2}).
%\end{proof} 

In case $E$ is smooth, Theorem \ref{thm:cont1} also follows from Corollary \ref{prop:bir1}.

\begin{remark}  \label{rem:5.7} Suppose that $\dim X = 3$. The condition $H^0(E,T_E) \ne 0$ holds in the following cases $E$ smooth).
\begin{enumerate} 
\item $E$ be a del Pezzo surface and $K_E^2 \geq 6. $ 
\item $E$ is a rational ruled surface.
\item $E$ is a ruled surface over an elliptic curve, and the vector field on the elliptic curve lifts to $E$.
\item $E$ is a ruled surface over $C$, and $E = \mathbb P(V)$ with a rank $2$-vector bundle $V$ such that $h^0(V^* \otimes V) \geq 2$,
since by the relative Euler sequence, the relative vector fields are computed by 
$$ h^0(E,T_{E/C})  = h^0(C,V^* \otimes V) - 1 $$ 

\end{enumerate} 
\end{remark} 

If $H^0(E,T_E) = 0$, things get more involved, we restrict ourselves to dimension three. 
The key is the following

\begin{proposition} \label{prop:small} 
Let $X$ be a Calabi-Yau threefold, $\varphi: X \to Y$ be a birational morphism to a normal compact complex (Moishezon) space $Y$, whose
exceptional locus is a smooth rational curve $C$. 
Then $H^1(X,T_X \otimes \Omega^1_X) \ne 0$.
\end{proposition}

\begin{proof} We argue by contradiction and assume to the contrary that $$H^1(X,T_X \otimes \Omega^1_X) = 0.$$
By a theorem of Laufer \cite[Thm.4.1]{La81}, the normal bundle $N_C = N_{C/X} $ has the following form 
\begin{enumerate}
\item $ N_C = \sO_C(-1) \oplus \sO_C(-1)$;
\item $N_C = \sO_C \oplus \sO_C(-2)$;
\item $N_C = \sO_C(1) \oplus \sO_C(-3).$
\end{enumerate}
Moreover, $y_0 = \varphi(C)$ is a hypersurface singularity. \\
We claim that 
\begin{equation} \label{CLAIM}  h^0(Y,R^1\varphi_*(T_X \otimes \Omega^1_X)) \geq 5. \end{equation} 
To prove Claim (\ref{CLAIM}), we use the inequality
$$ h^0(Y,R^1\varphi_*(T_X \otimes \Omega^1_X)) \geq h^1(C,T_X \otimes \Omega^1_X \vert C).$$
In fact, all cohomology classes in $H^1(C,T_X \otimes \Omega^1_X \vert C)$ extend to all infinitesimal neighborhoods, 
since $H^2(C, (N^*_C)^{\otimes k} \otimes T_X \otimes \Omega^1_X \vert C) = 0$ for all $k$. 
Since 
$$  h^1(C,T_X \otimes \Omega^1_X \vert C) = 5 $$ in Case (2) and
$$  h^1(C,T_X \otimes \Omega^1_X \vert C) = 7 $$ in Case (3), we need only to consider
Case (1). In this case, 
$$  h^1(C,T_X \otimes \Omega^1_X \vert C) = 4. $$
Here we need to consider the second infinitesimal neighborhood $C_2$, defined by the ideal $\sI_C^2$,
and use the inequality 
$$ h^0(Y,R^1\varphi_*(T_X \otimes \Omega^1_X)) \geq h^1(C_2,T_X \otimes \Omega^1_X \vert C_2).$$
The right hand side appears in the cohomology sequence
$$ 0 \to H^0(C,N^*_C \otimes T_X \otimes \Omega^1_X \vert C) \to H^0(C_2,T_X \otimes \Omega^1_X \vert C_2) \buildrel {\alpha} \over {\to}  H^0(C,T_X \otimes \Omega^1_X)  \to $$
$$ \to H^1(N_C^* \otimes T_X \otimes \Omega^1_X) \to H^1(C_2,T_X \otimes \Omega^1_X) \vert C_2).$$
By \cite[Thm.3.2]{La81}, a sufficiently small neighborhood of $C \subset X$ is biholomorphic to a small neighborhood of the zero section of the normal bundle $N_C$, hence 
$\alpha$ is surjective. Since 
$$h^1(N_C^* \otimes T_X \otimes \Omega^1_X) = 4,$$
Claim (\ref{CLAIM}) also holds in Case (1).
\vskip .2cm 
Since $\varphi$ is small, the sheaf $\varphi_*(T_X \otimes \Omega^1_X)$ is reflexive, hence 
$$ \varphi_*(T_X \otimes \Omega^1_X) = (T_Y \otimes \Omega^{[1]}_Y)^{**} =: \sF.$$ 
Since we assume $H^1(X,T_X \otimes \Omega^1_X) = 0$, the Leray spectral sequence yields
$$ H^1(Y,\sF) = 0;$$
further, the edge morphism 
$$ \mu: E_2^{0,1} \to E_2^{2,0} $$
is injective, hence by Claim (\ref{CLAIM}),
\begin{equation} \label{eq:comp}  h^2(Y,\sF) \geq h^0(Y,R^1\varphi_*(T_X \otimes \Omega^1_X)) \geq 5. \end{equation} 
We now compute $H^2(Y,\sF)$ in a different way to obtain a contradiction. By Serre duality,
$$ H^2(Y,\sF) \simeq {\rm Ext}^1(\sF,\sO_Y).$$ 
We will use the Grothendieck spectral sequence,  with $E_2$-terms
$$ E_2^{p,q} = H^p(Y,\sExt^q(\sF,\sO_Y)), $$
converging to ${\rm Ext}^{p+q}(\sF,\sO_Y)$. 
Notice that $$ H^p(Y, \sExt^0(\sF,\sO_Y)) = H^p(Y,\sHom(\sF,\sO_Y)) = H^p(Y,\sF),$$
since $\sF \simeq \sF^*$. 
Thus, introducing the edge morphism $\delta: E_2^{0,1} \to E_2^{2,0} $, the spectral sequence together with the vanishing $E_2^{1,0} = 0$ yields
$$ H^2(Y,\sF) \simeq \ker \delta.$$
Since ${\rm Ext}^2(\sF,\sO_Y) = H^1(Y,\sF) = 0$, necessarily $E_3^{2,0} = 0$, i.e., $E_2^{2,0} = \im \  \delta$,  so 
the morphism $\delta$ is surjective. Since $E_2^{2,0} = H^2(Y,\sF) \simeq \ker \delta$, we obtain 
\begin{equation} \label{eq:cont2}  2h^2(Y,\sF) = \dim E_2^{0,1} =  h^0(Y,\sExt^1(\sF,\sO_Y)).\end{equation} 
The sheaf  $\sExt^1(\sF,\sO_Y)$ being supported on $y_0$, we need to compute its length at $y_0$. 
Recalling that $y_0$ is a hypersurface singularity, 
$\Omega^1_Y = \Omega^{[1]}_Y$ by a theorem of Kunz \cite[Cor. 9.8]{Kun86}, hence
$$ \sF = \sHom(\Omega^1_Y,\Omega^1_Y).$$ 
For our local computation, we may assume $Y$ itself to be a hypersurface in $\mathbb C^4$. 
We consider the cotangent sequence
$$ 0 \to N^*_{Y /\mathbb C^4} \to \Omega^1_{\mathbb C^4} \vert Y \to \Omega^1_Y \to 0 ,$$
which after possibly shrinking $Y$ reads
$$ 0 \to \sO_Y \to \sO_Y^{\oplus 4} \to \Omega^1_Y \to 0. $$ 
Tensoring by $T_Y$ gives
\begin{equation} \label{eq:HR}  0 \to T_Y \to T_Y^{\oplus 4} \to T_Y \otimes \Omega^1_Y \to 0. \end{equation} 
The sheaf $T_Y \otimes \Omega^1_Y$ is clearly torsion free, as seen directly from the exact sequence (\ref{eq:HR}),  but possibly not reflexive. 
To see the difference, introduce the quotient $\mathcal R = \sF / (T_Y \otimes \Omega^1_Y)$ which is supported on $y_0$.  Dualizing the 
resulting exact sequence
$$ 0 \to T_Y \otimes \Omega^1_X \to \mathcal F \to \mathcal R \to 0 $$
 gives 
$$ \sExt^1(\mathcal R,\sO_Y ) \to \sExt^1(\sF,\sO_Y) \to \sExt^1(T_Y \otimes \Omega^1_Y, \sO_Y).$$
Since $\mathcal R $ is supported on $y_0$, we have  $\sExt^1(\mathcal R,\sO_Y ) = 0$, and thus it suffices to estimate $h^0(Y, \sExt^1(T_Y \otimes \Omega^1_Y,\sO_Y)$. 
Dualizing the exact sequence (\ref{eq:HR}) yields the exact sequence
$$ (\Omega^1_Y)^{\oplus 4} \buildrel {\alpha} \over {\to} \Omega^1_Y \to \sExt^1(T_Y \otimes \Omega^1_Y, \sO_Y) \to \sExt^1(\Omega^1_Y,\sO_Y)^{\oplus 4}.$$ 
Since $\alpha$ is simply the map $$ \id \otimes (df)^*: \Omega^1_Y \otimes \sO_Y^{\oplus 4} \to \Omega^1_Y \otimes \sO_Y = \Omega^1_Y,$$
where $f$ is the equation for $Y \subset  \mathbb C^4$, it follows
$$ {\rm coker} \alpha = \sO_{y_0}^{\oplus 4}.$$ 
Dualizing the cotangent sequence and using the same argument gives also
$$ \sExt^1(\Omega^1_Y,\sO_Y) = \sO_{y_0}.$$ 
Hence in total,
$$ h^0(Y,\sExt^1(T_Y \otimes \Omega^1_Y,\sO_Y)) \leq 8,$$
and consequently by Equation (\ref{eq:cont2}), 
$$ h^2(Y,\sF) =  \frac{1}{2} h^0(Y,\sExt^1(\sF,\sO_Y)) \leq 4.$$
This contradicts Inequality (\ref{eq:comp}), completing the proof of Proposition \ref{prop:small}. 

\end{proof}

\begin{corollary} \label{cor:small} Let $X$ be a Calabi-Yau threefold and $C \subset X$ be a smooth rational curve with normal bundle $N_C$. Assume either that 
$N_C = \sO_C(-1) \oplus \sO_C(-1)$ or that $N_C = \sO_C \oplus \sO_C(-2)$ and that $C$ is an isolated curve in the sense of \cite{Re83}, i.e., $C$ does not move in $X$. 
Then $H^1(X,T_X \otimes \Omega^1_X) \ne 0$.
\end{corollary}

\begin{proof} By Proposition \ref{prop:small}, it suffices to prove that $C$ is contractible. In the first case this is Grauert's criterion \cite{Gra62}; 
in the second case we apply a theorem of Reid \cite[Cor. 5.6]{Re83}.
\end{proof} 

We now apply Corollary \ref{cor:small} to compute $H^1(X,T_X \otimes \Omega^1_X)$.

\begin{theorem} \label{cont2}  Let $\varphi: X \to Y$ be a primitive contraction of the Calabi-Yau threefold with exceptional divisor $E$. Assume that $\dim \varphi(E) = 0$ and that one of the following conditions
holds.
\begin{enumerate} 
%\item $H^0(E,T_E) \ne 0$; 
\item $E$ is smooth;
\item $K_E^2 \geq 6$;
\item $E$ is normal, rational and contains a smooth contractible rational curve, e.g., $E$ carries a birational contraction of an extremal ray;
\item $E$ is normal and irrational;
\item $E$ is non-normal.
\end{enumerate}
Then 
$$ H^1(X, \Omega^1_X \otimes T_X) \ne 0.$$
\end{theorem} 

\begin{proof} %(1) This is Theorem \ref{thm:cont1}. \\
%\vskip .2cm \noindent 
(1)  If $E$ is smooth and if $K_E^2 \geq 6$, the claim follows from Theorem \ref{thm:cont1}, combined with Remark \ref{rem:5.7}, since $E$ is a del Pezzo surface. 
If $E$ is smooth with $K_E^2 \leq 7$, then we may choose a $(-1)-$curve $C \subset E$. Then, using the normal bundle sequence for $C \subset E \subset X$, it is immediate that 
$C$ has normal bundle $N_{C/X} = \sO_C(-1) \oplus \sO_C(-1) $. Hence Corollary \ref{cor:small} applies. 
\vskip .2cm \noindent 
(2) 
By a theorem of Gross, \cite[proof of 5.8]{Gr97} and Wilson, \cite[p.620-624]{Wi97},
there exists an open neighborhood $ U = U_0 \subset X$ of $E_0 := E$ and a deformation 
$$ \pi: \mathcal U \to \Delta $$
over the unit disc
and a divisor $\mathcal E \subset \mathcal U$ 
such that $X_0 = \pi^{-1}(0)$, such that $\mathcal E \cap X_0 = E_0$ and such that - after possibly shrinkling $\Delta$ - 
the divisor $E_t = \mathcal E \cap X_t$ is smooth. Since the normal bundle $N_{E/X}$ is negative, so does $N_{E_t/U_t},$ hence $E_t$ is contractible. 
 Thus, we obtain a family $\phi_t: U_t \to V_t$ contracting fiberwise the divisor $E_t.$ Notice that 
\begin{equation} \label{eq:triv} K_{U_t} \simeq \sO_{U_t} \end{equation}  
for all $t.$ 
In fact, we may choose $U_t$ such that $E_t$ is a deformation retract of $U_t.$ Hence the restriction
$$  H^2(U_t,\mathbb Z) \to H^2(E_t, \mathbb Z)$$
is an isomorphism. Since $H^q(E_t,\mathcal O_{E_t}) = 0 $, the restriction 
$$ {\rm Pic}(U_t) \to {\rm Pic}(E_t)$$
is an isomorphism, too. Since $K_X \vert E \simeq \mathcal O_E,$ it follows that $K_{U_t} \vert E_t \equiv 0,$ hence $K_{U_t} \vert E_t \simeq \mathcal O_{E_t}.$ 
Hence Equation (\ref{eq:triv}) follows. 

\vskip .2cm We assume now that $K_E^2 \geq 6$. 
Using the inclusion $$ N^*_{E_t/U_t} \subset \Omega^1_{U_t},$$ 
and observing $N^*_{E_t/U_t} \simeq \sO_{E_t}(-K_{E_t})$, we have an inclusion
$$  H^0(E_t, T_{E_t}) = H^0(E_t, T_{E_t} \otimes N^*_{E_t/U_t} \otimes \sO_{E_t}(K_{E_t}))\to $$
$$ \to  H^0(E_t, T_{U_t} \otimes \Omega^1_{U_t} \vert E_t \otimes {\sO}_{E_t}({K_E}_t)).$$
Since $K_{E_t}^2 = K_E^2 \geq 6$, we conclude 
Observe that, using the conormal sheaf sequence and the triviality of $K_{U_t}$ that 
$$ H^0(E_t, T_{U_t} \otimes \Omega^1_{U_t} \vert E_t \otimes {\sO}_{E_t}(K_{E_t})) = \ne 0.$$
By semi-continuity, 
$$ H^0(E, T_U \otimes \Omega^1_U \vert E \otimes \sO_E(K_E)) \ne 0.$$
By Serre duality,
$$ H^2(E, T_X \otimes \Omega^1_X \vert E) \ne 0,$$
and we conclude by Lemma \ref{jun28}. 

\vskip .2cm \noindent (3) Suppose now that $K_E^2 \leq 5$ and $E$ is normal and rational (but singular).
Then by \cite{HW81}, $E$ is either a rational surface with only ADE singularities or an elliptic cone; hence in our case, the first alternative holds. Let $C \subset E$ be a smooth contractible rational curve; $C$ is a 
$\mathbb Q$-divisor in $E$, but possibly not Cartier. 
The conormal sheaf $N^*_{C/E}$ is of the form 
$$ N^*_{C/E} = \sO_C(a) \oplus \sT$$
with $a \geq 0$ and $\sT$ a torsion sheaf, supported on $C \cap {\rm Sing}(E)$. 
Consider the conormal sheaf sequence
$$ 0 \to N^*_{E/X} \vert C \to N^*_{C/X} \to N^*_{C/E} \to 0.$$ 
Since  $N^*_{E/X} \vert C = -K_E \vert C$ is ample, either $a = 1$ and $N^*_{C/X} = \sO_C(1) \oplus \sO_C(1)$, or $a = 0 $, $\sT$ is supported on one point with one-dimensional stalk and 
 $N^*_{C/X} = \sO_C(2) \oplus \sO_C$. In both cases we conclude by Corollary \ref{cor:small}.

\vskip .2cm \noindent  (4) If $E$ is normal with a non-rational singularity, then, as already mentioned, $E$ is an elliptic cone. In this case, $H^0(E,T_E) \ne 0$, and we conclude by Theorem \ref{thm:cont1}. 

\vskip .2cm \noindent (5) Suppose finally that $E$ is non-normal. 
In case $y_0$ is not a hypersurface singularity, $K_E^2 = 7$ by \cite[Thm. 5.2]{Gr97}. This case is settled by (2). Thus we may assume that $y_0$ is a hypersurface singularity. 
Then  
$$ 1 \leq K_E^2 \leq 3$$
by \cite[p.201]{Gr97}, see also \cite[p.620]{Wi97}. Non-normal del Pezzo surfaces being classified by Reid 
 \cite{Rei94}, we show case by case that $H^0(E,T_E) \ne 0$. 
Then we conclude by Theorem \ref{thm:cont1}.
Actually, we have more informations on $E$, \cite[p.201]{Gr97} and \cite[p.620]{Wi97}. 
In fact, if $K_E^2 = 1$, then $E $ is a hypersurface in the weighted projective space $\mathbb P(3,1,1,1)$ with explicit equation, 
in case $K_E^2 $ the surface $E$ is a hypersurface in $\mathbb P(2,1,1,1)$ and if $K_E^2 = 3$, then $E$ is a cubic. The normalization $\widetilde E$  in the last
two cases are Hirzebruch surfaces, while the normalization in the first case is $\mathbb P_2$. 

We give two examples. First assume that $K_E^2 = 3$; so $E$ is a cubic surface in $\mathbb P_3$. The nonnormal locus $N$ (with the structure given by the conductor ideal) is a line in $E$. 
Then the tangent sheaf sequence reads
$$ 0 \to T_E \to T_{\mathbb P_3} \vert E \to \sI_N \otimes N_{E/\mathbb P_3} \to 0$$
(actually it suffices that the image of $  T_E \to T_{\mathbb P_3} \vert E$ is contained in $ \sI_N \otimes N_{E/\mathbb P_3} $). 
Now $h^0(E,T_{\mathbb P_3} \vert E) = 15$ and $h^0(E,N_{E/\mathbb P_3}) = h^0(E, \sO_E(3)) = 19$. 
Since the image of the restriction map $H^0(E,\sO_E(3)) \to H^0(N,\sO_N(3)) $ is surjective, it has dimension $5$, hence 
$h^0(E,\sI_N \otimes  N_{E/\mathbb P_3}) = 14$, hence $h^0(E,T_E) \ne 0$. 

Second, consider the case $K_E^2 = 1$, hence $\widetilde E \simeq \mathbb P_2$. Let $\eta: \widetilde E \to E$ be the normalization map;
the non-normal locus $N$ of $E$ is a line. 
We consider the subcase when the preimage $\widetilde N$ is a {\it smooth} conic; it might also be a line pair or a double line. The degree of $\eta \vert \widetilde N \to N$ is two. 
Then, $H^0(\widetilde E,T_{\widetilde E }) = H^0(\widetilde N, T_{\widetilde E} \vert \widetilde N)$ and the latter contains $H^0(\widetilde N, T_{\widetilde N})$. Thus there are three vectors fields
on $\widetilde E$, tangent to $\widetilde N$ and one of them is invariant under the double cover $\widetilde N \to N$ (by considering $T_{\widetilde N} \to \eta^*(T_N)$).

%The epimorphisms 
%$$ H^2(E_t, \Omega^1_{U_t} \otimes N_{E_t/U_t}) \to H^2(E_t, \Omega^1_{E_t} \otimes N_{E_t/U_t}) \simeq H^0(E_t,T_{E_t}) \ne 0$$
%and
%$$ H^2(E_t, \Omega^1_{U_t} \otimes T_{U_t} \vert E_t) \to H^2(E_t, \Omega^1_{U_t} \otimes N_{E_t/U_t})$$
%show that
%$$ H^2(E_t, \Omega^1_{U_t} \otimes T_{U_t} \vert E_t) \ne 0 $$
%for general $t.$ 
%By semi-continuity, 
%$$ H^2(E_0, \Omega^1_{U_0} \otimes T_{U_0} \vert E_0) \ne 0,$$
%i.e.,
%$$H^2(E, \Omega^1_X \otimes T_X \vert E) \ne 0. $$
%Now it follows easily (see the arguments in the next proposition) that 
%$$ H^2(X, \Omega^1_X \otimes T_X) \ne 0, $$
%hence 
%$$H^1(X, \Omega^1_X \otimes T_X) \ne 0. $$

\end{proof}

\begin{remark} The remaning open case in Theorem \ref{cont2} is the following:  $E$ is normal, but singular, rational with only ADE-singularities containing no contractible smooth rational curve
 and $K_E^2 \leq 5$.
Then necessarily $\rho(E) \leq 2$, otherwise $E$ carries a birational contraction of an extremal ray. The way to treat this open case would be to show that there is a global deformation $X_t$ of $X$ 
(not only a local deformation of a neighborhood of $E$), such that $E$ deforms to a smooth del Pezzo surface $E_t$ and then to conclude again by semicontinuity.

\end{remark} 

We now turn to the case that $E$ is contracted to a curve $C$. We already know that $T_X$ has a first order deformation if  $E$ is smooth with $H^0(E,T_E) \ne 0$, e.g., if $C \simeq 
\mathbb P_1$. The case that $C$ is an elliptic curve is easy as well, provided $c_3(X) \ne 4$: 

\begin{proposition} Let $X$ be a Calabi-Yau threefold,  $\varphi: X \to Y$ be a primitive contraction with smooth exceptional divisor $E$ such that $C = \varphi(E) $ is a curve. 
If $g(C) = 1$ and if $c_3(X) \ne 4$, 
then $H^1(X,\Omega^1_X \otimes T_X) \ne 0.$ 
\end{proposition}

\begin{proof} 
By Lemma \ref{jun28} and Remark \ref{rem:stability},   it suffices to show 
$$ H^2(E, \Omega^1_X \otimes T_X \vert E) \ne 0. $$ 
By Serre duality, this is equivalent to 
$$ H^0(E, \Omega^1_X \otimes T_X \vert E \otimes K_E) \ne 0. $$
Using the tangent bundle sequence, it suffices to show that 
$$ H^0(E, \Omega^1_X \otimes T_E \otimes K_E) = H^0(E,\Omega^1_X \otimes \Omega^1_E) \ne 0.$$
By the cotangent sequence, we obtain an exact sequence
$$ H^0(E,T_E) \to  H^0(E, \Omega^1_X \otimes \Omega^1_E) \to $$
$$ \to H^0(E,\Omega^1_E \otimes \Omega^1_E) \buildrel{\delta} \over {\to}  H^1(E,N_E^* \otimes \Omega^1_E) \simeq H^1(E,T_E). $$
Since we may assume $H^0(E,T_E) = 0 $ by Theorem \ref{thm:cont1} and since  clearly $H^2(E,T_E) = 0,$ Riemann-Roch shows, using $g(C) = 1,$ that 
$$ H^1(E,T_E) = 0.$$ 
Since $$ H^0(E,\Omega^1_E \otimes \Omega^1_E) \ne 0,$$
our claim follows. 

\end{proof}

\begin{remark} If $g(C) \geq 2,$ these simple arguments do no longer work (provided $\Omega_Y^{[1]}$ is $H$-stable). 
The difficulty is that $\dim H^1(E,T_E) = 6(g-1),$ assuming $H^0(E,T_E) = 0$, whereas 
$$\dim H^0(E,\Omega^1_E \otimes \Omega^1_E) = 3(g-1). $$ One would need to show that the
connecting map 
$$ \delta:  H^0(E,\Omega^1_E \otimes \Omega^1_E) \to H^1(E,N_E^* \otimes \Omega^1_E) \simeq H^1(E,T_E)$$
is not injective.
Actually, this statement can still be sharpened. In fact, assuming as always that $H^0(E,T_E) = 0$, then the sequence
$$ 0 \to T_{E/C} \to T_E \to \varphi^*(T_C) \to 0 $$
splits. Thus it suffices to show that 
$$ H^0(E,\Omega^1_X \vert E \otimes \varphi^*(T_C)) \ne 0. $$
This comes down to show that the canonical morphism
$$ \epsilon: H^0(E,\Omega^1_E \otimes \varphi^*(T_C)) \to H^1(E,N_E^* \otimes \Omega^1_E) $$
is not injective. 
Notice also that both vector spaces have the same dimension $3(g(C)-1)$. 
\end{remark} 

We treat next the case $g(C) \geq 2$ and the case that $E$ is singular by a more sophisticated method.

\begin{theorem} \label{thm:ruled}
Let $X$ be a Calabi-Yau threefold, $\varphi: X \to Y$ a primitive divisorial contraction contracting the exceptional divisor $E$ to a curve $C$. 
Then $H^1(X,T_X \otimes \Omega^1_X) \ne 0$ unless (possibly) $E$ is one of the following surfaces
\begin{enumerate}
\item $E$ is a normal singular surface and all singular fibers are double lines;
\item $E$ is a non-normal surface, $C \simeq \mathbb P_1$, but the normalization of $E$ is irrational. 
\end{enumerate} 
\end{theorem}

\begin{proof}
Recall that $p := \varphi_{\vert E}: E \to C$ is a conic bundle over the smooth curve $C$. 
We will also use the finer classification of $E$, due to Wilson  \cite{Wi92}, \cite{Wi93}, \cite{Wi97}. 
In fact, consider a singular fiber $E_c$ of $p$ which is a line pair $E_c = C_1 \cup C_2$ (with $C_1 \ne C_2$). 
Then there are three possibilities:
\begin{itemize} 
\item $E$ is smooth along $C_1 \cup C_2$;
\item $E$ is normal along along $C_1 \cup C_2$ and has an $A_n$-singularity at the intersection point $C_1 \cap C_2$ and is smooth elsewhere;
\item $E$ is non-normal, the normalisation is a $\mathbb P_1$-bundle over a smooth curve $\tilde C$ which is a double cover over $C$ 
and unramified over $c$. The singular locus of $E$ meets $C_1$ exactly in the intersection point $C_1 \cap C_2$. 
\end{itemize}

(A) Assume first that the general fiber of $p$ is irreducible, but $p$ has a reducible fiber $E_c$, and fix an irreducible component $B \simeq \mathbb P_1$. 
Then the conormal sheaf sequence 
$$ 0 \to N^*_{E/X} \vert B \to N^*_{B/X} \to N^*_{B/E} \to 0 $$
implies - together with the equation  $E \cdot B = -1$ - that 
$N^*_{B/X} $ is either $\sO_B(1) \oplus \sO_B(1)$ or $\sO_B(2) \oplus \sO_B$. 
In both cases $B$ is contractible; in the first case  by Grauert's criterion \cite{Gra62}, in the second case we apply again \cite[Cor. 5.6]{Re83}, using our assumption that the general fiber of $p$ is irreducible, hence
$B$ does not move. We also use the sequence 
$$ 0 \to N^*_{E_c/E} \vert B \simeq \sO_B \to N^*_{B/E} \to N^*_{B/E_c} \to 0 $$
to conclude that either $N^*_{B/E} \simeq \sO_B(1)$ or that $N^*_{B/E} \simeq \sO_B \oplus \sO_{x}$, where $\sO_{x}$ is a sheaf supported on the singularity of $E_c$ with 
one-dimensional stalk at $x$. 
Once we know that $B$ is contractible, we conclude by Proposition \ref{prop:small}. 
\vskip .2cm \noindent
(B) Next we consider the case $g(C) \geq 1$. 
By \cite[1.2,1.3]{Gro97a}, \cite[p.631 ff]{Wi97} there exists a flat family 
$$ \pi: \mathcal X \to \Delta$$
of Calabi-Yau threefolds $X_t$ over the unit disc 
with the following properties
\begin{enumerate}
\item $X_0 = \pi^{-1}(0) \simeq X$;
\item there is a relative crepant contraction
$ \Phi: \mathcal X \to \mathcal Y$ over $\Delta$ with $\Phi \vert X_0 = \varphi$;
\item $\varphi_t = \Phi \vert X_t: X_t \to Y_t$ is small;
\item $\varphi_t$ contracts the deformations of the finitely many fibers of $ E \to C$ which deform
to $X_t$.
\end{enumerate} 
By Theorem \ref{thm:small},  $H^1(X_t, T_{X_t} \otimes \Omega^1_{X_t}) \ne 0$,  hence $H^1(X,T_X \otimes \Omega^1_X) \ne 0$ by
 semi-continuity. \\

\vskip .2cm \noindent 
(C) We next  consider the case that $C \simeq \mathbb P_1$ and that all fibers of $p: E \to C$ are irreducible. If $E$ is smooth, $p$ is a $\mathbb P_1$-bundle, and we are done by Lemma \ref{lem:stability2}, Theorem \ref{thm:cont1} in
connection with Remark \ref{rem:5.7}. 
If $E$ is singular, then $E$ is normal and by (A), we may assume that the only singular fibers are double lines. This case is ruled out by assumption.

\vskip .2cm \noindent 
(D) Assume finally that the general fiber of $p$ is reducible. Then there is a double cover $\tilde C \to C$ such that the fiber product $\tilde E \to \tilde C$ is a $\mathbb P_1$-bundle. 
If $g( C) > 0$, then 
by \cite[p.635]{Wi97} and \cite[p.294]{Gro97a}, the general deformation $X_t$ of $X$ carries a small contraction. Hence 
$$H^1(X_t,T_{X_t} \otimes \Omega^1_{X_t}) \ne 0$$
by Theorem \ref{thm:small}, and we conclude by semicontinuity. \\
If $g(\tilde C) = 0$, then by \cite[p.635]{Wi97}, the non-normal locus of $E$ is a $(-1,-1)$-curve and we conclude by Proposition \ref{prop:small}
Finally, the case $g(\tilde C) > 0$ and $g(C) = 0 $ is ruled out by assumption.

%$E$ will deform in any deformation of $X$ and for the general deformation $X_t$, the surface $E_t$ will have a fibration $p_t: E_t \to C_t$ with irreducible general
%fiber. Furthermore, $E_t$ is contractible and by the previous steps, $H^1(X_t,T_{X_t} \otimes \Omega^1_{X_t}) \ne 0$. We then conclude by semi-continuity. \\
 
\end{proof} 

\begin{remark}  The exceptional case (b) might be ruled out as follows. Consider again a general deformation $X_t$ of $X$. Then $E$ deforms to a rational surface which is a conic
bundle over $\mathbb P_1$; see  \cite[p.635]{Wi97}. One might expect that not all singular fibers are double lines, hence $H^1(X_t,T_{X_t} \otimes \Omega^1_{X_t}) \ne 0$ 
and we conclude again by semicontinuity. Thus the difficulty is that in all deformation $E_t \subset X_t$, we land in the exceptional case (1). 

\end{remark}

We finally consider small contractions $\varphi: X \to Y$. 

\begin{theorem} \label{thm:small}  Let $X$ be a Calabi-Yau threefold and $\varphi: X \to Y$ be a small contraction. Then 
$H^1(X,T_X \otimes \Omega^1_X) \ne 0$. 
\end{theorem} 

\begin{proof} By \cite{La81}, \cite{Pin83}, \cite[Thm.5.5]{Mor85}, see also \cite{Fri86}, any fiber $F$ of $\varphi$ has the form $F = \cup C_j$ with smooth rational curves $C_j$. All $C_j$ have normal bundle 
$$ N_{C/X} = \sO_C(a) \oplus \sO_C(b) $$
with $(a,b) = (-1,-1),(0,-2),(1,-3).$ Moreover, there is at most one curve with $(a,b) = (1,-3)$.  
Thus we conclude by Corollary \ref{cor:small}, applied to some component $C_j$ whose normal bundle is not of type $(1,-3)$ - unless $F = C_1$ is irreducible with normal bundle of type $(1,-3)$. Then we apply Proposition \ref{prop:small}.

\end{proof}

Any small contraction $\varphi: X \to Y$ gives rise to a flop $h: X \dasharrow X^+$ with a (smooth) Calabi-Yau threefold $X^+$, \cite[2.4]{Ko89}. 
We finally relate the deformations of $T_X$ to those of $T_{X^+}$.

\begin{proposition} \label{prop:flop} 
Let $h: X \dasharrow X^+$ be a flop of the Calabi-Yau threefold $X$, induced by the small contraction $\varphi: X \to Y.$  Then 
$$ H^1(X,T_X \otimes \Omega^1_X)  \simeq  H^1(X,T_{X^+} \otimes \Omega^1_{X^+}). $$
Moreover, every positive-dimensional deformation of $T_X$ over (a germ of ) an irreducible reduced complex space $S$  induces canonically a positive-dimensional deformation of
 $T_{X^+}$ over $S.$ 

\end{proposition}

\begin{proof} 
In order to show the first claim, it suffices to show that 
$$ H^2X,T_X \otimes \Omega^1_X)  \simeq  H^2(X,T_{X^+} \otimes \Omega^1_{X^+}). $$
The flop being induced by a small contraction $\varphi: X \to Y$, we let  
$\varphi^+: X^+ \to Y$ denote the associated flopped small morphism. 
The Leray spectral sequence gives 
$$  H^2(X,T_X \otimes \Omega^1_X)  \simeq  H^2(Y,\varphi^+_*(T_X \otimes \Omega^1_X)) $$
and   
$$  H^2(X^+,T_{X^+} \otimes \Omega^1_{X^+})  \simeq  H^2(Y,\varphi^+_*(T_{X^+} \otimes \Omega^1_{X^+})) .$$
The sheaves  $\varphi_*(T_X \otimes \Omega^1_X)$ and $\varphi^+_*(T_{X^+} \otimes \Omega^1_{X^+})$ are reflexive and
isomorphic outside a finite set. Hence they are isomorphic on all of $Y$, and the claim follows. \\
As to the second claim, let $\mathcal E$ be a flat deformation of $T_X$ over $X \times S.$ The locally free sheaf $\mathcal E$ induces a coherent sheaf
$\mathcal E^+$ over $X^+ \times S$, such that $(\mathcal E^+)_{\vert X^+ \times \{0\}}\simeq T_{X^+}.$ 
In particular,  $(\mathcal E^+) _{\vert X^+ \times \{0\}} $ is locally free and so does  $(\mathcal E^+)_{\vert X^+ \times \{s\}}$ for small $s$. 
Hence $\mathcal E^+$ is flat over $S.$

\end{proof}

\begin{corollary}  Let $X$ and $X'$ be birationally equivalent Calabi-Yau threefolds. Then
$$ H^1(X,T_X \otimes \Omega^1_X)  \simeq  H^1(X,T_{X'} \otimes \Omega^1_{X'}). $$
Moreover, every positive-dimensional deformation of $T_X$ induces canonically a posi\-tive-dimensional deformation of
 $T_{X'}.$
\end{corollary} 

\begin{proof} It suffices to remark that any birational map between (smooth) Calabi-Yau threefolds is a sequence of
flops, \cite{Ko89}. 
\end{proof} 

%\begin{remark}  Suppose the smooth Calabi-Yau threefold $X$ contracts an irreducible divisor $E$ to a curve $C$, which is 
%automatically smooth. Then by \cite[sect. 4]{Wi97}, in the general deformation the divisor $E$ will disappear; actually only finitely many
%curve in the fibers of $E \to C$. This means that the contraction on $X$ is a limit of  small contractions on the deformations $X_t$. 
%Consequently, if $H^1(X_t,T_{X_t} \otimes \Omega^1_{X_t}) \ne 0 $, hence by semi-continuity, $H^1(X,T_X \otimes \Omega^1_X) \ne 0$. 
%Thus one can work with $X_t$, i.e., performing some flops and might hope to come into a better situation. 
%\end{remark} 

\section{Higher Dimensions}
\setcounter{lemma}{0}

We finish the paper with some results in higher dimensions: hypersurfaces in projective space, and products, the latter being important for the correct set-up of Question \ref{Q}. 

\begin{theorem} \label{thm:n} Let $X \subset \mathbb P_{n+1}$, $n \geq 2$, be a smooth hypersurface of degree $d \geq 2$. 
Then $T_X$ has a genuine first order deformation. 

\end{theorem} 

\begin{proof} Since $H^1(X,\sO_X) = 0$, it suffices to show that $H^1(X,T_X \otimes \Omega^1_X) \ne 0$. 
We use the cohomology sequence 
$$ 0 \to H^0(X,T_X \otimes \Omega^1_X) \to H^0(X,T_{\mathbb P_{n+1}} \vert _X \otimes \Omega^1_X) \to H^0(X,\Omega^1_X(d)) \to $$
$$ \to  H^1(X,T_X \otimes \Omega^1_X) .$$
The Euler sequence in combination with the vanishing $H^1(X,\Omega^1_X(1)) = 0$ (use the cotangent sequence for $X \subset \mathbb P_{n+1})$
yields 
$$ h^0(X,T_{\mathbb P_{n+1}} \vert _X \otimes \Omega^1_X) = 1.$$
Thus $$ H^0(X,\Omega^1_X(d)) \subset H^1(X,T_X \otimes \Omega^1_X) $$
and it suffices to observe that $ H^0(X,\Omega^1_X(d)) \ne 0$, which is clear since $\Omega^1_{\mathbb P_{n+1}}(2)$ is generated by global sections. 
\end{proof} 

Concerning products, we first consider the case of two factors. 

\begin{proposition} \label{product1} Let $X = X_1 \times X_2$ be a projective manifold with $\dim X_j \geq 1$. 
Then $$H^1(X,\sEnd_0(T_X)) = 0 $$ if and only if the following conditions hold for $j = 1,2$. 
\begin{enumerate} 
\item $H^0(X_j,T_{X_j}) = 0 $;
\item $q(X_j) = 0 $;
\item $ H^1(X_j, T_{X_j} \otimes \Omega^1_{X_j}) = H^1(X_j,\sEnd_0(T_{X_j})) = 0$. 

\end{enumerate}
In particular, $\dim X_j \geq 2$ for $j = 1,2$. 

\end{proposition}

\begin{proof}  Let $p_j: X \to X_j$ denote the projections. Then 
$$ T_X \otimes \Omega^1_X \simeq \big( p_1^*(T_{X_1} \otimes \Omega^1_{X_1}) \big)  \oplus 
\big(p_1^*T_{X_1} \otimes p_2^*\Omega^1_{X_2} \big)  \ \oplus $$
$$ \oplus \ 
\big(p_2^*T_{X_2 } \otimes p_1^*\Omega^1_{X_1} \big) \oplus 
\big(p_2^*(T_{X_2} \otimes \Omega^1_{X_2}) \big). $$
Using the K\"unneth formula, a direct computation shows that the conditions (1),(2) and (3) are equivalent to $h^1(X, T_X \otimes \Omega^1_X) = q(X)$, i.e., 
$h^1(X, \sEnd_0(T_X)) = 0$. In particular, $\dim X_j = 1$ is impossible. 

\end{proof} 

Inductively, we obtain

\begin{corollary} \label{product3}
Let $X = \Pi_{j=1}^m X_j$ be a projective manifold. 
Then $$H^1(X,\sEnd_0(T_X)) = 0 $$ if and only if the following holds for all $j$.

\begin{enumerate}
\item $H^0(X_j,T_{X_j}) = 0 $;
\item $q(X_j) = 0 $;
\item $ H^1(X_j,\sEnd_0(T_{X_j})) = 0$.
\end{enumerate}

\end{corollary}

\bibliographystyle{amsalpha}

\bibliography{biblio}

\end{document}